\documentclass[11pt, longbibliography, reqno]{amsart}     
\usepackage{amsfonts, amssymb}	
\usepackage[a4paper, margin=2.7cm]{geometry}
\usepackage{amsmath}
\usepackage[colorlinks,linkcolor=blue]{hyperref}
\usepackage[nameinlink]{cleveref}
\usepackage{mathtools}
\usepackage{xcolor}
\usepackage{graphicx}
\usepackage{physics}
\usepackage[dvipsnames]{xcolor}
\usepackage{epstopdf}
\usepackage{amsthm}
\usepackage{verbatim}
\usepackage{pdfpages}
\newcommand{\q}[1]{``#1''}

\usepackage{subcaption}

\theoremstyle{plain}
\newtheorem{theorem}{Theorem}[section]
\newtheorem{corollary}[theorem]{Corollary}
\newtheorem{proposition}[theorem]{Proposition}
\newtheorem{lemma}[theorem]{Lemma}
\theoremstyle{definition}
\newtheorem{definition}[theorem]{Definition}
\newtheorem{example}[theorem]{Example}
\theoremstyle{remark}
\newtheorem{remark}[theorem]{Remark}
\newtheorem*{note*}{Note}
\newtheorem*{remark*}{Remark}

\usepackage{tikz}
\usepackage{tikz-cd}
\usepackage{caption}

\newcommand{\eqn}[0]{\begin{array}{rcl}}
	\newcommand{\eqnend}[0]{\end{array} }  	

\newcommand{\T}{\mathbb{T}}
\newcommand{\D}{\mathbb{D}}
\newcommand{\C}{\mathbb{C}}

\usepackage{bbm} 
\let\svthefootnote\thefootnote
\newcommand\freefootnote[1]{%
  \let\thefootnote\relax%
  \footnotetext{#1}%
  \let\thefootnote\svthefootnote%
}

\title[Higher-dimensional singularities of rational inner functions]{Higher-dimensional singularities of rational inner functions}
\author{Leonora Krajina}
\address{Department of Mathematics, Stockholm University, 106 91 Stockholm, Sweden}
\email{nora.kraj3@gmail.com}
\date{\today}                                           
\subjclass[2020]{32A08, 32A40}
\keywords{Rational inner functions, singularities, derivative integrability, slice matrix}

\begin{document}
\maketitle

    \begin{abstract}
        We study singularities of rational inner functions (RIFs) $\phi$ in three and more variables. We focus on curve and higher-dimensional singularities of $\phi$ in $\T^n$ and 
        we show that, under certain conditions, derivative integrability conditions of partial derivatives of $\phi=\frac{q}{p}$ stay constant along irreducible components of the zero set of $p$, and that the presence of vertical line singularities does not affect global integrability. This answers Question 2 in \cite{bickel2022singularities}. Those results are then generalized to $n$ variables. Finally, we consider slice matrices, which we use to obtain derivative integrability conditions of unidirectional compositions $\phi^N$. We present several new examples, including an RIF with two vertical line singularities and one with vertical surface singularities.
    \end{abstract}

   \section{Introduction}
    A holomorphic function is called \textit{inner} if it is bounded on a domain in $\C^n$ and has boundary values of modulus one almost everywhere. This paper deals with a special type of such functions, defined on the unit polydisk
    \[\D^n=\{z=(z_1,\dots,z_n)\in\mathbb{C}^n \,:\,|z_j|<1, \, j=1,\dots,n\}\subset\mathbb{C}^n, \]
    with potential singularities on its distinguished boundary, the $n$-torus
    \[\T^n=\{z=(z_1,\dots,z_n)\in\mathbb{C}^n \,:\,|z_j|=1, \, j=1,\dots,n\}\subset\mathbb{C}^n, \]
    called rational inner functions. Their denominator polynomials $p\in\mathbb{C}[z_1,\dots,z_n]$ are \textit{stable} with regards to the unit polydisk $\D^n$, i.e. have the property that 
    \[\mathcal{Z}_p=\{z\in\C^n\,:\, p(z)=0\}\cap \D^n=\varnothing.\]

    Then, given two polynomials $p,q\in\mathbb{C}[z_1,\dots,z_n]$ with no common factors, where $p$ is stable, the function $\phi=\frac{q}{p}$ is called a \textit{rational inner function (RIF)} if $|\phi(z)|=1$ for almost every $z\in\T^n$. Rational inner functions are by definition holomorphic on $\D^n$, with $|\phi(z)|<1$ on $\D^n$, and
    \vspace{-0.5pc}
    \[|\phi^*(z)|=1\quad \text{for }z\in\T^n,\]
    where $\phi^*(z)$ is the radial limit of $\phi$, i.e.
    \[\phi^*(z)=\lim_{r\to1^-}\phi(rz):=\lim_{r\to1^-}\phi(rz_1,\dots,rz_n).\]
    This limit indeed exists at all points of the $n$-torus, by \cite[Theorem C]{knese_nontanglimits}. 
    
    \vspace{1mm}
    Rational inner functions play a significant role in multiple areas of mathematics, where the most prominent one is in multivariable function and operator theory. They appear as solutions to Nevanlinna-Pick interpolation problems \cite{agler_pick_int} and approximations of Schur functions on the polydisk \cite{rudin1969function}, playing an important role in the construction of Agler decompositions on the bidisk \cite{bickel_agler}. Moreover, stable polynomials appear in the study of dynamical systems \cite{jury}, \cite{dynamics_sola_tullydoyle}.

    \vspace{1mm}
    There is an alternate description of RIFs, using the notion of a \textit{reflection polynomial} $\tilde{p}$ of $p$, which is defined as 
    \vspace{-0.5pc}
    \[\tilde{p}(z_1,\dots,z_n)=z_1^{m_1}\cdots z_n^{m_n}\,\overline{p\bigg(\frac{1}{\overline{z_1}},\dots,\frac{1}{\overline{z_n}}}\bigg),\]
    where $m=\deg p=(m_1,\dots,m_n)$ is the polydegree of $p$. Here each $m_j$ denotes the degree of $p$ in the variable $z_j$. The following classic result by Rudin and Stout \cite{rudin-stout} (also found in \cite[Chapter 5]{rudin1969function}) and independently by Pfister \cite{pfister}, then states:

    \begin{theorem}
        Given polynomials $p,q\in\mathbb{C}[z_1,\dots,z_n]$ with no common factors, where $p$ is stable, the function $\phi=\frac{q}{p}$ is a rational inner function (RIF) if and only if $q(z)$ is of the form
        \[az^m\, \tilde{p}(z)\]
        for $a\in\T$ and $m\in\mathbb{Z}^n_{\geq 0}$ (written in multi-index notation). Moreover, every RIF $\phi$ is of the form 
        \vspace{-0.5pc}
        \[\phi(z)=az^m\frac{\tilde{p}(z)}{p(z)}\]
        for some stable $p$ that has no common factors with $\tilde{p}$.
    \end{theorem}

    Rational inner functions are precisely the multivariable generalization of finite Blaschke products (see \cite{garcia2016finiteblaschkeproductssurvey} for an extensive overview), which are rational, inner and holomorphic on the unit disk $\D$. Unlike the one-variable Blaschke products, even though they have radial limits at all points of $\T^n$, RIFs are possibly not continuous on the whole of $\overline{\D^n}$. Unless $\tilde{p}(z)=\text{const}\cdot p(z)$, the function $\phi$ will have boundary singularities at points where $p(z)=0$. For example, the \q{favourite example}
    \vspace{-0.5pc}
    \[\phi(z_1,z_2,z_3)=\frac{\tilde{p}(z_1,z_2,z_3)}{p(z_1,z_2,z_3)}=\frac{3z_1z_2z_3-z_1z_2-z_1z_3-z_2z_3}{3-z_1-z_2-z_3}\]
    has a boundary singularity at the point $\zeta=(1,1,1)$. Even though both $p$ and $\tilde{p}$ vanish at the point $\zeta$, it is not possible to factorize and cancel out common factors. To make matters simpler, we will be concerned only with RIFs of the form $\phi=\frac{\tilde{p}}{p}$. The following properties are well-known and simple to check.
    
    \begin{proposition}{\label{p and tilde p similarities}}
    For $p\in\C[z_1,\dots, z_n]$ stable, and $\tilde{p}$ its reflection polynomial,
        \begin{itemize}
            \item[i)] The polydegrees of $p$ and $\tilde{p}$ are equal. 
            \item[ii)] $\mathcal{Z}_p\cap\T^n=\mathcal{Z}_{\tilde{p}}\cap\T^n$.
            \item[iii)] $|p|=|\tilde{p}|$ on $\T^n$.
        \end{itemize}
    \end{proposition}
    
    Since $p$ and $\tilde{p}$ are of the same polydegree, we will say that an RIF $\phi=\frac{\tilde{p}}{p}$ is of polydegree $m$ if $p$ (and equivalently, $\tilde{p}$) is of polydegree $m$. 
    
    \vspace{2mm}
    Singularities of rational inner functions can be examined using different methods. One way would be to study the geometry of the unimodular level sets of the RIF, like in \cite{Bickel_2020} or more recently in \cite{bickel2022singularities}. We could also look at finite iterations $\phi^N$ of $\phi$ and observe the changes in the singular sets as we iterate. This is something that is investigated in detail in Section \ref{Compositions} via the characterization of the partial derivative integrability around the singularities. More specifically, for an RIF $\phi$ in $n$ variables, the main interest will be to find $\mathfrak{p}\geq1$ such that $\frac{\partial \phi}{\partial z_j}\in L^{\mathfrak p}(\T^n)$, for $j=1,\dots,n$, but we will mostly investigate the case when $j=n$. An extensive analysis had been done for two-variable RIFs in \cite{Bickel_2017}, while \cite{bickel2022singularities} dealt with RIFs in three and more variables and displayed behaviour previously unseen in the two variable case. As we increase the number of variables, the zero set $\mathcal{Z}_{p}$ on the torus can increase in dimension, so instead of point singularities, the problem becomes one of curve singularities, which are generally harder to work with. Narrowing the question to a special type of RIFs linear in the last variable, and with the singular set comprised purely of isolated points, previously known results from two variables have been extended to more variables in \cite{sola_2023}. 
    
    \vspace{1mm}
    The structure of the paper is the following. In Section \ref{Prerequisites}, we give a brief overview of the main integrability results from \cite{Bickel_2017}, \cite{bickel2022singularities} and \cite{sola_2023}. In Section \ref{section on curve singularities}, we move to $(m_1,m_2,1)$-degree RIFs in three variables with curve singularities and relate the zero set of $\phi$ on $\T^3$ to the zero set of an auxiliary function $\rho_{\phi}$ on $\T^2$, in a similar manner as was done in \cite{sola_2023}. We present a series of results for partial integrability bounds for such RIFs, extending those in \cite{bickel2022singularities} and \cite{sola_2023}. For instance, we answer Question 2 in \cite{bickel2022singularities} by showing that the presence of vertical line singularities does not affect global integrability. Further, in Section \ref{section higher dimensions}, we extend these results to higher-dimensional RIFs linear in the last variable, with singular sets of codimension one.
    
    Section \ref{Compositions} turns to an investigation of iterations of such RIFs, by describing the singular set of $\phi^N$ on $\T^n$ for $N\in\mathbb{N}$, in terms of the singular set of $\phi$ on $\T^n$. After showing that the singular sets are preserved under iteration, we extract integrability conditions for $\frac{\partial\phi^N}{\partial z_n}$ from those for $\frac{\partial\phi}{\partial z_n}$.
    
    Finally, Section \ref{Examples} is dedicated to constructing an original example of a three-variable rational inner function with two vertical line singularities.
    
    \section{Preliminaries}{\label{Prerequisites}}
     We start the paper with an overview of known integrability results for two-variable RIFs, and then move to the $n$-variable ones. The reason for this separation of the two cases lies in the following theorem extracted from \cite[Theorem 2.4]{aglermccarthystankus2005toralalgebraicsetsfunction}:
    \begin{theorem}{\label{agler-mccarthy-stankus2006}}
        For an irreducible RIF $\phi=\frac{\tilde{p}}{p}$ in $n$ variables, 
        $\dim (\mathcal{Z}_p\cap\T^n)\leq n-2$.
    \end{theorem}

    \noindent
    This means that two-variable RIFs can only have point singularities, while the singular sets of three-variable ones may also contain curves. See \cite[Example 5.3.]{bickel2022singularities} for an example of an RIF with both an isolated singularity and a curve of singularities. In the same manner, four-variable RIFs may have surface, curve and/or point singularities and so on. In this section we will review results on RIFs containing only point singularities. 

    \subsection{Two-variable RIFs}
    Singularities of two-variable RIFs have been thoroughly investigated in \cite{Bickel_2017} through multiple viewpoints, but mostly through the geometry of $\mathcal{Z}_p$ on the faces of the bidisk $\D\times\T$ and $\T\times\D$ near a singular point $(\xi_1,\xi_2)\in\T^2$. Without loss of generality, we restrict to $\D\times\T$ by fixing $\zeta_2\in\T$ near $\xi_2$. Let $\alpha_1(\zeta_2),\dots \alpha_k(\zeta_2)\in\D$ denote the points where $\tilde{p}(\alpha_i(\zeta_2),\zeta_2)=0$, i.e. $(\alpha_i(\zeta_2),\zeta_2)\in \mathcal{Z}_{\tilde{p}}\cap (\D\times \T)$. It was shown that $\mathcal{Z}_{\tilde{p}}\cap (\D\times \T)$ approaches the singularity $(\xi_1,\xi_2)$ such that there exists an even number $K_1$ satisfying 
    \[\min_{1\leq i \leq k} \big(1-|\alpha_i(\zeta_2)|\big)\approx|\xi_2-\zeta_2|^{K_1}\]
    for all $\zeta_2\in\T$ sufficiently close to $\xi_2$. The integer $K_1$ was called the  \textit{local $z_1$-contact order of $\phi$ at $(\xi_1,\xi_2)$.} 
    
    Taking the maximum $K_1$ over all of the singularities of $\phi$ gives us the notion of the \textit{$z_1$-contact order of $\phi$}. This in turn characterizes the \textit{critical integrability index} $\mathfrak p^*$ of $\frac{\partial \phi}{\partial z_1}$. This is the number $\mathfrak p^*:=\sup \mathfrak p$, where $\mathfrak p \geq 1$ (see Remark \ref{remark p in Lp^1}), such that $\frac{\partial \phi}{\partial z_1}\in L^{\mathfrak p}(\T^2)$, in the following way. By \cite[Theorem 4.1.]{Bickel_2017}, for $1\leq \mathfrak p <\infty$, 
    \[\frac{\partial \phi}{\partial z_1}\in L^{\mathfrak p}(\T^2) \iff K_1<\frac{1}{\mathfrak{p}-1}.\]

    We can define the $z_2$-contact order of $\phi$ in a similar manner, starting with restricting to $\T\times\D$ by fixing $\zeta_1\in\T$ near $\xi_1$. The even integer $K_2$ would then serve to obtain an analogous result of the partial integrability, i.e. that for $1\leq \mathfrak r <\infty$, 
    \[\frac{\partial \phi}{\partial z_2}\in L^{\mathfrak r}(\T^2) \iff K_2<\frac{1}{\mathfrak{r}-1}.\]

    \vspace{0.5pc}
    \noindent
    As it turns out, in two variables we have $\mathfrak p^* = \mathfrak r^*$, that is, by \cite[Theorem 4.3]{Bickel_2020},
    \[\frac{\partial \phi}{\partial z_1}\in L^{\mathfrak p}(\T^2) \iff \frac{\partial \phi}{\partial z_2}\in L^{\mathfrak p}(\T^2).\]
    However, in higher dimensions, this result does not hold and critical integrability indices can differ from a variable to another. See \cite[Example 3.1.]{bickel2022singularities}. Moreover, a notion of contact order hasn’t been systematically developed for RIFs in three and more variables, but there are other ways of characterizing the derivative integrability.

    \subsection{RIFs in three and more variables}

    Previously in \cite{bickel2022singularities}, three-dimensional rational inner functions with isolated singularities were investigated in detail. The paper \cite{sola_2023} focused on such RIFs as well, but in the general higher dimensional ($n\geq 3$) setting. Specifically, in order to characterize the integrability of the partial derivative in the last variable, the analysis was restricted to $(m_1,\dots,m_{n-1},1)$-degree irreducible RIFs. In this case, the root of the defining polynomial $p(z_1,\dots,z_{n-1},\tau)$ can be expressed as a rational function $\tau=\psi(z_1,\dots z_{n-1})$, yielding a single-valued parametrization of $\tau\in\T$. When the degree in $z_n$ is two or higher, this is typically no longer possible.

    \vspace{0.5mm}
    Now, the linearity in the last variable allows us to write
    \begin{align}{\label{general formula phi}}
    \phi(z)=\frac{\tilde{p}_1(z_1,\dots,z_{n-1})\cdot z_n+ \tilde{p}_2(z_1,\dots,z_{n-1})}{p_1(z_1,\dots,z_{n-1})+ p_2(z_1,\dots,z_{n-1}) \cdot z_n}
    \end{align}
    for $z=(z_1,\dots,z_n)$ and where $p_j,\tilde{p}_j\in\C[z_1,\dots,z_{n-1}]$, for $j=1,2$. By explicitly writing $z_n$ in terms of the other variables, we obtain a parametrization of $\mathcal{Z}_{\tilde{p}}$,
    \[\tilde{p}_1(z_1,\dots,z_{n-1})\cdot z_n+ \tilde{p}_2(z_1,\dots,z_{n-1}) =0 \iff z_n=-\frac{\tilde{p}_2(z_1,\dots,z_{n-1})}{\tilde{p}_1(z_1,\dots,z_{n-1})}\]
    whenever $\tilde{p}_1\neq 0$. This is always the case for degree $(m_1,m_2,1)$ RIFs with isolated zero sets; see the discussion in the beginning of Section \ref{section on curve singularities} for a clear explanation of why. After setting 
    \[z_n=\psi^0(z_1,\dots,z_{n-1}):=-\frac{\tilde{p}_2(z_1,\dots,z_{n-1})}{\tilde{p}_1(z_1,\dots,z_{n-1})},\] 
    an auxiliary function $\rho_{\phi}$ is then defined as 
    \begin{align*}
    \rho_{\phi}(z_1,\dots,z_{n-1}) :&= 1-|\psi^0(z_1,\dots,z_{n-1})|^2 \\
    &=\frac{|\tilde{p}_1(z_1,\dots,z_{n-1})|^2-|\tilde{p}_2(z_1,\dots,z_{n-1})|^2}{|\tilde{p}_1(z_1,\dots,z_{n-1})|^2}.
    \end{align*}
    On the torus, the zero set of the function $\rho_{\phi}$ is closely connected to the zero set of $p$. Specifically, for a point $\zeta=(\zeta_1,\dots,\zeta_n)\in\T^n$, we have that 
    \begin{align}
    p(\zeta)=0 \iff \rho_{\phi}(\zeta_1,\dots,\zeta_{n-1})=0 \iff (|\tilde{p}_1|^2-|\tilde{p}_2|^2)(\zeta_1,\dots,\zeta_{n-1})=0.
    {\label{equation in beginning, phi equiv with rho_phi}}
    \end{align}
    This is a consequence of \cite[Lemma 2]{sola_2023}, the proof of which we now outline. By fixing  $\hat{\zeta}=(\zeta_1,\dots,\zeta_{n-1})\in\T^{n-1}$, a one-variable rational function 
    \[\phi_{\hat{\zeta}}(z_n)=\phi(\hat{\zeta},z_n)\]
    is obtained, which by \cite[Theorem C]{knese_nontanglimits}
    has unimodular boundary values at every point $z_n\in\T$. Then it follows that $\phi_{\hat{\zeta}}(z_n)$ has to be either a Möbius transformation or constant, in which case the unimodularity would imply it to be equal to some point on the torus $\T$. Now the Lemma proves that in this latter case when $\phi_{\hat{\zeta}}(z_n)=\alpha\in\T$, we have that $|\tilde{p}_1(\hat{\zeta})|^2-|\tilde{p}_2(\hat{\zeta})|^2=0$. This however means that the whole vertical line $\{\hat{\zeta}\}\times\T$ is contained in the level set 
    \[C_{\alpha}=\{z\in\T:\tilde{p}(z)-\alpha p(z)=0\}.\]
    Now the Lemma invokes \cite[Theorem 3.3, Lemma 3.4]{clarkmeasuresrationalinner2021}, where the equality $\hat{\zeta}=(\tau_1,\dots,\tau_{n-1})$ comes from, for some singularity $(\tau_1,\dots,\tau_n)\in\T^n$ of $\phi$ (the result was actually shown in two variables, but it was nowhere directly used that $\hat{\zeta}\in\T$). The other direction is proven in detail in Lemma \ref{Z_p and Z_rho_phi (almost) the same}.

    \vspace{0.1pc}
    This result turns out to be crucial in the characterization of the partial derivative integrability of $\phi$ around its singularities (i.e. zeros of $p$), as it is enough to observe the behaviour of the function $\rho_{\phi}$ around the points restricted to their first $n-1$ variables. In the absence of contact order, we state the following result from \cite[Theorem 2.1, Section 3.1]{bickel2022singularities} and \cite[Lemma 2]{sola_2023}, which serves as an analogue for higher dimensions in the setting of $z_n$-linear RIFs.

    \begin{theorem}{\label{thm equiv Lp and integral}}
    For $1\leq\mathfrak p<\infty$,
        \[\frac{\partial\phi}{\partial z_n}\in L^{\mathfrak p}(\T^n) \iff \int_{\T^{n-1}} (1-|\psi^0(z_1,\dots,z_{n-1})|^2)^{1-\mathfrak p}\, |dz_1|\cdots|dz_{n-1}|<\infty. \]
    \end{theorem}
    The above theorem can be stated locally as well, that is, on small neighbourhoods around points of the zero set of $\rho_\phi$ on $\T^{n-1}$. In the following chapters, we will usually work with this local definition.
    \vspace{-1pc}

    \section{Curve singularities of $(m_1,m_2,1)$-degree RIFs}{\label{section on curve singularities}}

    We first study three-variable $(m_1,m_2,1)$-degree RIFs with curve singularities. As before, $\phi=\frac{\tilde{p}}{p}$ has no singularities inside the tridisk, which together with $|z_3|\leq 1$  on $\D$ gives $|p_1|\geq|p_2|$, by (\ref{general formula phi}). Passing to the boundary, we obtain the following inequality for points on $\T^2$:
    \vspace{-0.5pc}
    \begin{align}{\label{inequalities for p1 and p2}}
    |\tilde{p}_1|=|p_1|\geq|p_2|=|\tilde{p}_2|.
    \end{align}
    Therefore, if $p_1$ (or equivalently, $\tilde{p}_1$) vanishes at some point $(\xi_1,\xi_2)\in\T^2$, so does $p_2$ (and $\tilde{p}_2$). Then $z_3\in\T$ can be chosen arbitrarily, so $\mathcal{Z}_p\cap\T^3$ and $\mathcal{Z}_{\tilde{p}}\cap\T^3$ contain the whole vertical line $\{(\xi_1,\xi_2)\}\times\T$. Previously (see \cite{bickel2022singularities}), their presence has been problematic in the study of integrability of $\frac{\partial \phi}{\partial z_3}$, since the function
    \vspace{-0.5pc}
    \[\rho_{\phi}(z_1,z_2)=\frac{|\tilde{p}_1(z_1,z_2)|^2-|\tilde{p}_2(z_1,z_2)|^2}{|\tilde{p}_1(z_1,z_2)|^2}\]
    has a singularity at $(\xi_1,\xi_2)$ and is no longer analytic everywhere on $\T^2$. From now on, for a more concise notation, let $\sigma(z_1,z_2):=|\tilde{p}_1(z_1,z_2)|^2 - |\tilde{p}_2(z_1,z_2)|^2$.

    \begin{lemma}{\label{sigma real analytic}}
        The function $\sigma$ is real-analytic and non-negative on the closed bidisk $\overline{\D^2}$.
    \end{lemma}
    \begin{proof}
        The function $\sigma$ is defined as sum of real-valued polynomials on $\overline{\D^2}$. Such polynomials are real-analytic, so their sum is real-analytic as well. By the discussion in the beginning of Section \ref{section on curve singularities}, $|\tilde{p}_1|\geq |\tilde{p}_2|$. This inequality is preserved with squaring, so it immediately follows that $\sigma$ is non-negative as well.    
    \end{proof}

     In cases when $\rho_{\phi}$ turns out to have singularities and ceases to be analytic (Example \ref{Example 5.2.}), the following series of results will help us deal with them. It follows immediately from \cite[Lemma 3.4]{bickel2022singularities} that $\mathcal{Z}_{|\tilde{p}_1|^2}\cap \T^2$ consists of at most finitely many isolated points, as well as that:
     
    \begin{lemma}{\label{finitely VLS}}
        An irreducible $(m_1,m_2,1)$-degree RIF has at most finitely many vertical line singularities. 
    \end{lemma}

    \noindent
    We can now classify the zero set of $p$ on $\T^3$ in terms of the zero set of $\rho_{\phi}$ on $\T^2$.
    
    \begin{lemma}{\label{Z_p and Z_rho_phi (almost) the same}}
        Let $\phi$ be an irreducible $(m_1,m_2,1)$-degree RIF. Then the following holds. 
        \begin{enumerate}
            \item[(i)] If $\phi$ has no vertical line singularities, $\rho_{\phi}(\zeta_1,\zeta_2)=0$ if and only if $(\zeta_1,\zeta_2,\tau)$ is a singularity of $\phi$ for some $\tau\in\T$.
            \item[(ii)] If $\phi$ has $k$ vertical line singularities $\{\xi^1\}\times\T,\dots,\{\xi^k\}\times\T$, then $\rho_{\phi}(\zeta_1,\zeta_2)=0$ if and only if $(\zeta_1,\zeta_2,\tau)\in\mathcal{Z}_p\cap\T^3\,\backslash\big( \{(\xi^1,\dots,\xi^k\}\times\T\big)$.
        \end{enumerate}  
    \end{lemma}
    \begin{proof}
        (i)  Since the zero set of $p$ contains no vertical lines, neither does the zero set of $\tilde{p}$. By the observations made in the beginning of this section, this means that $|\tilde{p}_1|^2$ is non-vanishing on $\T^2$, and therefore bounded below on $\T^2$. Thus,
        \vspace{-0.5pc}
        \[\rho_{\phi}(z_1,z_2)=\frac{\sigma(z_1,z_2)}{|\tilde{p}_1(z_1,z_2)|^2}\]
       is analytic on $\T^2$. Therefore, $\mathcal{Z}_{\rho_{\phi}}=\mathcal{Z}_{\sigma}$, and we have 
       \begin{align*}
           \tilde{p}(\zeta_1,\zeta_2,\zeta_3)=0 &\iff \zeta_3\,\tilde{p}_1(\zeta_1,\zeta_2)+ \tilde{p}_2(\zeta_1,\zeta_2)=0 \\
           &\iff \zeta_3=-\frac{\tilde{p}_2(\zeta_1,\zeta_2)}{\tilde{p}_1(\zeta_1,\zeta_2)} \\
           & \Longrightarrow 1=|\zeta_3|^2=\frac{|\tilde{p}_2(\zeta_1,\zeta_2)|^2}{|\tilde{p}_1(\zeta_1,\zeta_2)|^2} \\
           &\Longrightarrow |\tilde{p}_1(\zeta_1,\zeta_2)|^2-|\tilde{p}_2(\zeta_1,\zeta_2)|^2=0 
        \end{align*}
        so $\rho_{\phi}(\zeta_1,\zeta_2)=0$. The other direction follows from the discussion after (\ref{equation in beginning, phi equiv with rho_phi}).
        \vspace{0.5pc}

        (ii) First, notice that (\ref{inequalities for p1 and p2}) implies $\mathcal{Z}_{|\tilde{p}_1|^2}\subset \mathcal{Z}_{\sigma}$. If $|\tilde{p}_1(\zeta_1,\zeta_2)|^2=0$, this means that $\tilde{p}_1(\zeta_1,\zeta_2)=0$, and therefore $\tilde{p}_2(\zeta_1,\zeta_2)=0$ as well. But then $|\tilde{p}_2(\zeta_1,\zeta_2)|^2=0$, so clearly $|\tilde{p}_1(\zeta_1,\zeta_2)|^2-|\tilde{p}_2(\zeta_1,\zeta_2)|^2=0$. Thus, $(\zeta_1,\zeta_2)\in\mathcal{Z}_{\sigma}$.
        
        But then $\mathcal{Z}_{\rho_{\phi}}=\mathcal{Z}_{\sigma}\,\backslash\, \{\xi^1,\dots,\xi^k\}$. By (i), $\mathcal{Z}_{\sigma}$ coincides with the projection of $\mathcal{Z}_p$ onto the first two coordinates. Now the claim follows immediately.   
    \end{proof}

\begin{example}{\label{Example 5.2.}}
    We take a look at the RIF from \cite[Example 5.2]{bickel2022singularities}:
	\begin{align*}
		\phi(z) & = \frac{1-z_1-z_1 z_2 +z_1^2z_2-2z_3-z_1z_3-z_1^2z_3+z_2z_3-z_1z_2z_3+4z_1^2z_2z_3}{4-z_1+z_1^2-z_2-z_1z_2-2z_1^2z_2+z_3-z_1z_3-z_1z_2z_3+z_1^2z_2z_3}.
	\end{align*}
	In the mentioned paper it has been shown that 
    \vspace{-0.5pc}
	\[\mathcal{Z}_p \,\cap \mathbb{T}^3 = \bigg\{(0,0,u) \, : \, u\in[-\pi,\pi]\} \,\cup\, \{s,\arg m(e^{is}),\pi) \,:\,s\in[-\pi,\pi] \, ,\, m(z)=\frac{3+z^2}{1+3z^2}\bigg\},\]
    written in terms of the arguments of the variables. This means that $\mathcal{Z}_p \,\cap \mathbb{T}^3$ contains the vertical line $\{(1,1)\}\times \mathbb{T}$. Using that the coefficients are real, we expand:
	\begin{align*}
		\rho_{\phi}(z_1,z_2) &  = \frac{|\tilde{p}_1(z_1,z_2)|^2-|\tilde{p}_2(z_1,z_2)|^2}{|\tilde{p}_1(z_1,z_2)|^2} = \frac{\tilde{p}_1(z_1,z_2)\tilde{p}_1(\bar{z}_1,\bar{z}_2) - \tilde{p}_2(z_1,z_2)\tilde{p}_2(\bar{z}_1,\bar{z}_2)}{\tilde{p}_1(z_1,z_2)\tilde{p}_1(\bar{z}_1,\bar{z}_2)}\\
		& = \frac{20+6\bar{z}_1^2-6\bar{z}_2-9\bar{z}_1^2\bar{z}_2+6z_1^2-z_1^2\bar{z}_2-6z_2-\bar{z}_1^2z_2-9z_1^2z_2}{24-2\bar{z}_1+6\bar{z}_1^2-5\bar{z}_2-2\bar{z}_1\bar{z}_2-8\bar{z}_1^2\bar{z}_2-2z_1+6z_1^2-z_1^2\bar{z}_2-5z_2-\bar{z}_1^2z_2-2z_1z_2-8z_1^2z_2}
	\end{align*}
	We're interested in the zero set of $\rho_{\phi}$ on the 2-torus, where \[z_j\bar{z}_j=|z_j|^2=1 \, ,\;j=1,2.\]
    Therefore, we can write each of the variables $z_j$ in terms of their arguments only, $z_j=e^{i\theta_j}$ for $\theta_j\in(-\pi,\pi]$. So as Lemma \ref{sigma real analytic} stated for the function $\sigma$, it is clear that $\rho_{\phi}$ is real-valued and non-negative at all points in which it is defined in $\overline{\D^2}$. Now we simplify
	\begin{align}
		\rho_{\phi}(z_1,z_2) & = \frac{20+\frac{6}{z_1^2}-\frac{6}{z_2}-\frac{9}{z_1^2 z_2}+6z_1^2-\frac{z_1^2}{z_2}-6z_2-\frac{z_2}{z_1^2}-9z_1^2z_2}{24-\frac{2}{z_1}+\frac{6}{z_1^2}-\frac{5}{z_2}-\frac{2}{z_1 z_2}-\frac{8}{z_1^2 z_2}-2z_1+6z_1^2-\frac{z_1^2}{z_2}-5z_2-\frac{z_2}{z_1^2}-2z_1z_2-8z_1^2z_2}\nonumber\\
        & = \frac{-9-6z_1^2-z_1^4+6z_2+20z_1^2z_2+6z_1^4z_2-z_2^2-6z_1^2z_2^2-9z_1^4z_2^2}{-(-4+z_1-z_1^2+z_2+z_1z_2+2z_1^2z_2)(-2-z_1-z_1^2+z_2-z_1z_2+4z_1^2z_2)} \nonumber\\
		& = \frac{(-3-z_1^2+z_2+3z_1^2z_2)^2}{(-4+z_1-z_1^2+z_2+z_1z_2+2z_1^2z_2)(-2-z_1-z_1^2+z_2-z_1z_2+4z_1^2z_2)}. \nonumber
	\end{align}
	To see where $\rho_{\phi}$ vanishes on the torus, we have to find the zero sets of the numerator and of the denominator, and then remove the latter set, i.e. the set of singularities of $\rho_{\phi}$. The denominator vanishes in two cases:
	\[-4+z_1-z_1^2+z_2+z_1z_2+2z_1^2z_2 = 0 \iff z_2=\frac{4+z_1+z_1^2}{1+z_1+2z_1^2}\]
    and 
     \[-2-z_1-z_1^2+z_2-z_1z_2+4z_1^2z_2=0 \iff z_2=\frac{2+z_1+z_1^2}{1-z_1+4z_1^2}.\]
	Using the fact that $|z_2|=1$, and taking the absolute value of the expressions above, the equations simplify to
    \vspace{-0.5pc}
	\begin{align*}
		z_1^4+z_1^3+6z_1^2+z_1+1=0,
	\end{align*}
	which has no solutions on $\T$, and 
    \vspace{-0.2pc}
	\[-2(z_1-1)^4=0,\]
	which has the solution $z_1=1$. After substituting into $p=0$, this gives $z_2=1$ as well. This is in accordance with Lemma \ref{Z_p and Z_rho_phi (almost) the same} and the zero set of $p$ on $\T^3$, since
	\begin{equation*}
		\rho_{\phi}(z_1,z_2)\cap\T^2 = \bigg\{\Big(z_1,\frac{3+z_1^2}{1+3z_1^2}\Big)\, :\, z_1\in\T\backslash\{1\}\bigg\}.
	\end{equation*} 
    \end{example}
    
	\begin{theorem}{\label{theorem_rho=p1-p2}}
		Let $\phi$ be a $(m_1,m_2,1)$-degree RIF with no vertical line singularities. Then for a point $\zeta=(\zeta_1,\zeta_2,\zeta_3)\in\mathcal{Z}_p\cap\T^3$, we have that $\frac{\partial\phi}{\partial z_3}\in L^{\mathfrak p}_{loc}(\T^3)$ at $\zeta$ if and only if 
		\[\int_{B_{\varepsilon}(\zeta_1,\zeta_2)}\sigma(z_1,z_2)^{1-\mathfrak
        p} \,dm(z_1,z_2) <\infty \]
		if and only if
		\[\int_{B_{\varepsilon}(\zeta_1,\zeta_2)} \rho_{\phi}(z_1,z_2)^{1-\mathfrak p} \,dm(z_1,z_2) <\infty \]
		for sufficiently small $\varepsilon>0.$
	\end{theorem}
	\begin{proof}
        The first equivalence follows from \cite[Theorem 2.1]{bickel2022singularities} the same way as Theorem \ref{thm equiv Lp and integral}. Secondly, since $\zeta\in\mathcal{Z}_p\cap\T^3$, Lemma \ref{Z_p and Z_rho_phi (almost) the same} gives $(\zeta_1,\zeta_2)\in \mathcal{Z}_{\rho_{\phi}}\cap\T^2$.
		Since $\phi$ contains no vertical line singularities on $\T^3$, the zero sets of $\rho_{\phi}$ and $\sigma$ are equal. The function $|\tilde{p}_1|^2$ is bounded from below on $\T^2$, so the integral 
        \[\int_{B_{\varepsilon}(\zeta_1,\zeta_2)}\bigg(\frac{\sigma(z_1,z_2)}{|\tilde{p}_1(z_1,z_2)|^2}\bigg)^{1-\mathfrak p} \,dm(z_1,z_2) \]
        converges precisely when
        \[\int_{B_{\varepsilon}(\zeta_1,\zeta_2)}{\sigma(z_1,z_2)}^{1-\mathfrak p} \,dm(z_1,z_2) \]
        does.
	\end{proof}

    \begin{remark}{\label{remark p in Lp^1}}
    For a general $n$-dimensional rational inner function $\phi$, it is important to note that around a  singularity $\zeta\in\T^n$, the partial derivative $\frac{\partial \phi}{\partial z_k}\in L^1_{loc}(\T^n)$ for all $k\in\{1,\dots,n\}$. This was shown for two variables in \cite[Proposition 4.7.]{Bickel_2017}, in a proof that used a restriction to a Blaschke product, on which the argument principle was then applied. The general $n$-dimensional case was then proven analogously in \cite[Lemma 3]{linusbergqvist}.
    \end{remark}
	
	\begin{lemma}{\label{lemma_rho}}
		Assume $\phi$ has k vertical line singularities at $\{\xi^1\}\times\T,\dots,\{\xi^k\}\times\T\subset\T^3$. Then at all points $\zeta\in\mathcal{Z}_p\cap\T^3$ not contained in $\{\xi^1,\dots,\xi^k\}\times\T$,
        \[\frac{\partial\phi}{\partial z_3}\in L^{\mathfrak p}_{loc}(\T^3) \iff \int_{B_{\varepsilon}(\zeta_1,\zeta_2)} \rho_{\phi}(z_1,z_2)^{1-\mathfrak p} \,dm(z_1,z_2) <\infty \]
		for sufficiently small $\varepsilon>0.$
	\end{lemma}
	\begin{proof}
		Such points $\zeta$ which are not contained in vertical line singularities, but still contained in $\mathcal{Z}_p\cap\T^3$, are either point singularities or contained in some curve singularity. In both cases, $\tilde{p}_1(\zeta_1,\zeta_2)\neq 0$, so the statement follows by the proof of Theorem \ref{theorem_rho=p1-p2}.
	\end{proof}

    We now turn to the question of global partial derivative integrability, namely, determining when $\frac{\partial \phi}{\partial z_3}\in L^{\mathfrak p}(\T^3)$. The following theorem stated in \cite[Theorem 6.6]{demailly2012complex}, will be of crucial importance in relating the local integrability properties to global ones. We refer the reader to \cite[Chapter 6.5]{tastybits} for the relevant terminology used in the theorem.
	
	\begin{theorem}{\label{Demailly}}
		Let M be a complex manifold with dim$_\mathbb{C}M=n$, let $(A,\xi)$ be an analytic germ of pure dimension $n-1$ and let $A_j$, $1\leq j\leq N$, be its irreducible components. Let $\mathcal{O}_{M,\xi}$ denote the ring of germs of holomorphic functions around $\xi\in M$. Then for every $f\in \mathcal{O}_{M,\xi}$ such that $f^{-1}(0)\subset (A,\xi)$, there is a unique decomposition $f=ug_1^{l_1}\cdots g_N^{l_N}$, where $u$ is an invertible germ, $g_j$ are irreducible germs, and $l_j$ is the order of vanishing of $f$ at any point $z\in A_{j, \text{reg}}\backslash \bigcup_{k\neq j} A_k$.
	\end{theorem}
	
	A subtle fact which will be useful to us is that the order of vanishing is constant along irreducible components of analytic germs. In our setting, this means that the integrability around an arbitrary point of a component of the zero set of $\rho_{\phi}$ on $\T^{n-1}$ governs the order of vanishing of $\rho_{\phi}$ around the point, which in turn dictates the integrability conditions around all other points of that irreducible component. This is the topic of the following theorem.
	
	\begin{theorem}{\label{all pts on component same Lp}}
        Let the zero set of $\sigma$ on $\T^2$ be of pure real dimension one so that $\mathcal{Z}_{\rho_{\phi}}\cap\T^2=\bigcup_{i=1} ^N \gamma_i$ for some $N\in\mathbb{N}$, where each $\gamma_i$ is a real smooth curve.
		Assume $\frac{\partial\phi}{\partial z_3}\in L^{\mathfrak p}_{loc}(\T^3)$ around a point $(\zeta_1,\zeta_2)\in\gamma_j\backslash \bigcup_{k\neq j} \gamma_k$, for some $j\in \{1,\dots,N\}$. Then $\frac{\partial\phi}{\partial z_3}\in L^{\mathfrak p}_{loc}(\T^3)$ at every $(z_1,z_2)\in\gamma_j\backslash \bigcup_{k\neq j} \gamma_k$.
	\end{theorem}
	\begin{proof}
		For an arbitrary $\gamma_k$ in the mentioned set of curves, the smoothness of the curve guarantees that $\rho_{\phi}$ has no singularities on $\gamma_k$, so $|\tilde{p}_1|$ is non-vanishing on $\gamma_k$. This is true for $\gamma_j$ as well, and by Lemma \ref{lemma_rho}, 
        \vspace{-0.5pc}
        \[\int_{B_{\varepsilon}(\zeta_1,\zeta_2)} \rho_{\phi}(z_1,z_2)^{1-\mathfrak p} \,dm(z_1,z_2) <\infty \]
		for small enough $\varepsilon>0$. By Theorem \ref{theorem_rho=p1-p2} this also means that for such  $\varepsilon$
		\begin{align}
            \int_{B_{\varepsilon}(\zeta_1,\zeta_2)}\sigma(z_1,z_2)^{1-\mathfrak p} \,dm(z_1,z_2) <\infty. \label{equation rho_phi in the proof}
         \end{align}
        For $(z_1,z_2)\in \T^2$ we have
        \vspace{-0.5pc}
        \[\begin{cases}
              z_1=e^{i\theta_1}, \quad\theta_1\in (-\pi,\pi)\\
              z_2=e^{i\theta_2}, \quad \theta_2\in (-\pi,\pi).
        \end{cases}\]
        Then $\sigma(z_1,z_2)=\sigma(e^{i\theta_1},e^{i\theta_2})$ is a real-analytic function since it's a composition of three real-analytic functions $\theta$, $e^{i\theta_1}$ and $e^{i\theta_2}$, for all $\theta_1,\theta_2\in(-\pi,\pi)$. 
        
		Let $A:=\mathcal{Z}_{\sigma}\cap\T^2$. Then $(\gamma_i)_i$ are exactly the sets of regular points of each of the irreducible components $A_i$ of $A$, i.e.
        \vspace{-0.5pc}
        \[A_{reg}=\bigcup_{i=1}^N A_{i,reg}=\bigcup_{i=1}^N \gamma_i.\]
        By the sentence preceding Lemma \ref{finitely VLS}, zeros of $|\tilde{p}_1|$ are isolated, so the set of singular points $A_{\text{sing}}$ of $A$ is comprised of finitely many points. The set $A$ is contained in $\T^2$, and since all $\gamma_i$ are curves, $A$ is of pure real dimension $1$.

        Denote the domain of $\sigma$ by $U:=(-\pi,\pi)^2$. This is an open and connected in $\mathbb{R}^2$, so there exists a domain $V\in\mathbb{C}^2$ such that $U\subset V$, and a unique holomorphic function $F:V\to\mathbb{C}$ such that $F\big|_U=\sigma$. This is a classic complexification property that can be found, for example, in \cite[Proposition 3.1.3]{tastybits}.

        Now let $B:=\mathcal{Z}_F$. Then $F$ is analytic at all $z=(z_1,z_2)\in V$, which makes $(B,z)\supset F^{-1}(0)$ an analytic germ, meeting one of the conditions of Theorem \ref{Demailly}.  Moreover, $F$ is holomorphic in all small neighbourhoods of $z$ contained in $V$, thus $F\in\mathcal{O}_{\mathbb{C}^2,z}$. It's important to notice that $(\zeta_1,\zeta_2)\in V$ as well. By Theorem \ref{Demailly}, there is a unique decomposition of $F$ as
        \[F=vf_1^{l_1}\cdots f_N^{l_N}\]
        where each $l_i$ is the order of vanishing of $F$ along irreducible component $B_i$ of $B$, $v$ is invertible and $f_j$ are irreducible germs. 

        Since $A$ is comprised of a finite union of disjoint real curves, if we take a small enough neighbourhood $W$ of $\T^2$ in $\mathbb{C}^2$, we will see that $\mathcal{Z}_{\sigma}\cap W$ will be made up of a finite number of (possibly not disjoint) real surfaces, which after complexification become complex curves. This comes from the uniqueness of $F$ and the identity theorem for holomorphic functions, as well as the fact that holomorphic functions in two variables cannot have isolated zeros. By taking small enough neighbourhoods $W_i$ of each $\gamma_i$, we will have $\gamma_i\subset W_i\subset B_i$ for each $i\in(1,\dots,N)$, and the order of vanishing of $\sigma$ will stay constant along each $\gamma_i$ the same way it was for $F$ along each $B_i$. We get the decomposition $f=ug_1^{l_1}\cdots g_N^{l_N}$, for $u$ invertible and $g_i$ irreducible germs, and $l_i$ the order of vanishing along each $\gamma_i$.
        
        One of the hypotheses of the theorem is that $\frac{\partial\phi}{\partial z_3}\in L^{\mathfrak p}_{loc}(\T^3)$ around $(\zeta_1,\zeta_2)$, which implies that (\ref{equation rho_phi in the proof}) holds. The order of vanishing of $\sigma$ at $(\zeta_1,\zeta_2)$, which is precisely $l_j$, can be determined from the series expansion of $\sigma$ around $(\zeta_1,\zeta_2)$. From the definition of order of vanishing, this means that the lowest power of a term in the power series expansion of $\sigma$ around some other point $(z_1,z_2)\in\gamma_j$ will be exactly $l_j$ as well. By applying a real-analytic change of coordinates, without loss of generality we can assume $\gamma_j=\{z_2=0\}$. Since every polynomial $f$ can be written as 
        \[f(z)=\sum_nf_n(z)\]
        where $f_n$ are homogeneous polynomials with all terms of degree $n$,
        then for every $(z_1,z_2)\in\gamma_j$ we can write
        \vspace{-0.5pc}
        \begin{align}{\label{equation change of variables}}
            \sigma(z_1,z_2)= z_1^{l_j}(\text{const + higher degree terms)} 
        \end{align}
        in its series expansion around the point $(z_1,z_2)$. Now it immediately follows that
        \[\int_{B_{\varepsilon}(z_1,z_2)}\sigma(w_1,w_2)^{1-\mathfrak p} \,dm(w_1,w_2) <\infty \]
        and thus 
        \[\int_{B_{\varepsilon}(z_1,z_2)} \rho_{\phi}(w_1,w_2)^{1-\mathfrak p} \,dm(w_1,w_2) <\infty \]
		for all $(z_1,z_2)\in\gamma_j$ and small enough $\varepsilon>0$. Applying Lemma \ref{lemma_rho} gives the wanted result.
    \end{proof}

    In practice, the theorem above tells us that for $\phi$ such that the zero set of $\rho_{\phi}$ on $\T^2$ contains no isolated points, we can read off the integrability indices along entire smooth components of zero curves just by knowing local integrability around points where $\rho_{\phi}$ is particularly easy to handle. Those include points that are not intersection points of two or more curves, are not the first two coordinates of a vertical line in the zero set of $\phi$, but can also mean such points where the series expansion of $\rho_{\phi}$ admits a simpler form. 

    The last results of this section will give us a way to make use of vertical lines or, depending on the objective, evade the need to know the precise integrability conditions of the partial derivatives around them. 
    
    \begin{corollary}{\label{corollary 3 equivalences around the vertical line and its component}}
        Let $\phi$ be an RIF as before with a vertical line $\{(\xi_1,\xi_2)\}\times\T\subset \mathcal{Z}_p\cap\T^3$. Let the zero set of $\rho_{\phi}$ on $\T^2$ contain only smooth curves $(\gamma_i)_i$ (as in Theorem \ref{all pts on component same Lp}). If there is a smooth curve $\Gamma\in\mathcal{Z}_{\sigma}\cap\T^2$ such that $\Gamma=\{(\xi_1,\xi_2)\}\cup\gamma_i$ for some $i$, then for an arbitrary point $(w_1,w_2)\in\Gamma\,\backslash \big(\bigcup_{k\neq j} \gamma_k \cup \{(\xi_1,\xi_2)\}\big)$ the following are equivalent:
        \begin{enumerate}
            \item[(i)] $\int_{B_{\varepsilon}(\xi_1,\xi_2)} \sigma(z_1,z_2)^{1-\mathfrak p} \,dm(z_1,z_2) <\infty$
            \item[(ii)] $\int_{B_{\varepsilon}(w_1,w_2)} \sigma(z_1,z_2)^{1-\mathfrak p} \,dm(z_1,z_2) <\infty$
            \item[(iii)] $\int_{B_{\varepsilon}(w_1,w_2)} \rho_{\phi}(z_1,z_2)^{1-\mathfrak p} \,dm(z_1,z_2) <\infty$
        \end{enumerate} 
    \end{corollary}
    \begin{proof}
        The equivalence between $(ii)$ and $(iii)$ was shown in Theorem $\ref{theorem_rho=p1-p2}$. The proof of the equivalence between $(i)$ and $(ii)$ follows the same steps as the proof of Theorem \ref{all pts on component same Lp} since $\sigma$ is a real-analytic function on a smooth curve $\Gamma$, which doesn't intersect any of the other smooth curves $\gamma_k$, for $k\neq i$. Thus it's possible to extend it to a holomorphic function, on which we can apply Theorem \ref{Demailly}.
    \end{proof}

    \begin{example}{\label{Example 5.2. global Lp}}
        Let's look at Example \ref{Example 5.2.} from before. We have
        \[\mathcal{Z}_p \,\cap \T^3 = \bigg\{(0,0,u) \, : \, u\in[-\pi,\pi]\} \,\cup\, \{s,\arg m(e^{is}),\pi) \,:\,s\in[-\pi,\pi] \, ,\, m(z)=\frac{3+z^2}{1+3z^2}\bigg\}\]
        so that 
        \vspace{-0.5pc}
        \[\mathcal{Z}_{\rho_{\phi}}\cap\T^2 = \bigg\{\Big(\xi_1,\frac{3+\xi_1^2}{1+3\xi_1^2}\Big)\, :\, \xi_1\in\T\backslash\{1\}\bigg\} \label{zero of rho}\]
        and $\mathcal{Z}_p \,\cap \T^3$ contains the vertical line $\{(1,1)\}\times\T$. On $\T^2$,  $\mathcal{Z}_{\rho_{\phi}}$ consists of a single smooth curve (Figure \ref{fig:5.2.rho_phi}), so the integrability index is the same along all points in $\mathcal{Z}_{\rho_{\phi}} \,\cap \T^2$. For example, take $(z_1,z_2)=(-1,1)\in\mathcal{Z}_{\rho_{\phi}} \,\cap \T^3$, or in other words, $(\theta_1,\theta_2)=(\pi,0)$, for $(z_1,z_2)=(e^{i\theta_1}, e^{i\theta_2})$. Notice that when $\theta_1$ is close to $\pi$, the point $\pi-\theta_1$ will be close to 0. Thus, after simplifying the trigonometric identities, around the point $(\pi,0)$ we have the expansions:
    \begin{align}
        \quad &\rho_{\phi}(\pi-\theta_1,\theta_2) = \nonumber\\
        & = \frac{20+12\cos2\theta_1-2\cos(2\theta_1+\theta_2)-12\cos\theta_2-18\cos(2\theta_1-\theta_2)}{24+4\cos\theta_1+12\cos2\theta_1-2\cos(2\theta_1+\theta_2)-10\cos\theta_2+4\cos(\theta_1-\theta_2)-16\cos(2\theta_1-\theta_2)}\nonumber \\ 
        & = (\theta_1-\theta_2)^2+\mathcal{O}(\|\theta\|^3 \nonumber\\
        \text{and }\nonumber\\
        & \hspace*{-1em} \quad\sigma(\pi-\theta_1,\theta_2)=16(\theta_1-\theta_2)^2+\mathcal{O}(\|\theta\|^3) \label{equation Example 5.2. sigma Lp}
    \end{align}
    for $(\theta_1,\theta_2)$ close to $(0,0)$. Here the change of variables as in (\ref{equation change of variables}) is easy to do, that is, by defining $\mu:=\theta_1-\theta_2$, we have $\rho_\phi(\pi-\theta_1,\theta_2)\approx \mu^2$ and $\sigma(\pi-\theta_1,\theta_2)\approx 16\mu^2$. The order of vanishing of the function $\rho_{\phi}(\pi-\theta_1,\theta_2)$ around the origin is two, so this means that $\frac{\partial\phi}{\partial z_3}\in L^{\mathfrak p}_{loc}(\T^3)$ around $(\pi,0,\tau)$, $\tau\in\T$ if and only if 
    \vspace{-0.4pc}
    \[2(1-\mathfrak p)>-1 \iff 2-2\mathfrak p>-1 \iff \mathfrak p<\frac{3}{2}.\]
    
    \noindent
    For the sake of comparison, take $(z_1,z_2)=(-i,-1)\in\mathcal{Z}_{\rho_{\phi}} \,\cap \T^2$, i.e. $(\theta_1,\theta_2)=(-\frac{\pi}{2},\pi)$. Then
    \begin{align*}
        \quad\rho_{\phi}&\Big(-\frac{\pi}{2}-\theta_1,\pi-\theta_2\Big)= (4\theta_1+\theta_2)^2+\mathcal{O}(\|\theta\|^3)
    \end{align*}
    for $(\theta_1,\theta_2)$ close to $(0,0)$. Now the change of variables $\mu=4\theta_1+\theta_2$ gives $\rho_{\phi}(-\frac{\pi}{2}-\theta_1,\pi-\theta_2)\approx\mu^2$ as well, so $\frac{\partial\phi}{\partial z_3}\in L^{\mathfrak p}_{loc}(\T^3)$ around $(\frac{\pi}{2},\pi,\tau)$, $\tau\in\T$ if and only if $\mathfrak p<\frac{3}{2}$, which is in accordance with Theorem \ref{all pts on component same Lp}. 
     \begin{figure}[h]
		\centering
		\includegraphics[width=4.5cm]{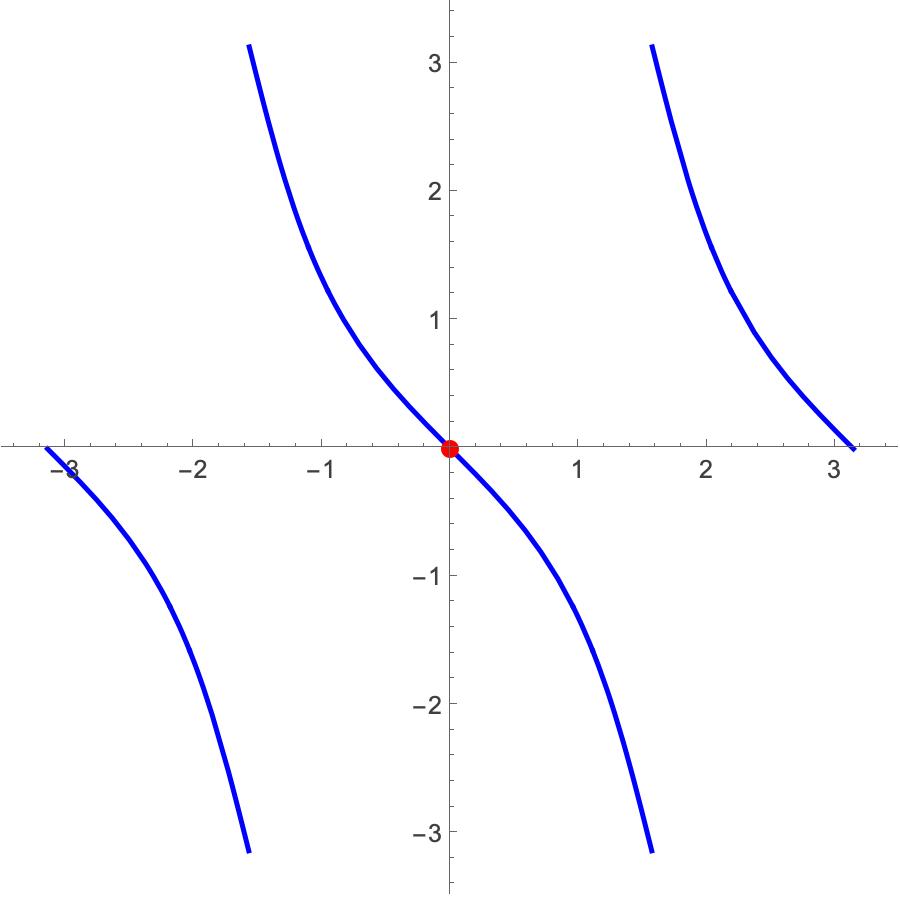}
		\caption{Example \ref{Example 5.2.}. Zero set of $\rho_{\phi}$ on $\T^2$, with the singular point $(0,0)$.}
		\label{fig:5.2.rho_phi}
	\end{figure}
    \end{example}

   \begin{theorem}{\label{situation better in vertical lines}}
       Let $\phi$ be an RIF that contains a vertical line singularity $\{(\xi_1,\xi_2)\}\times\T\subset \mathcal{Z}_p\cap\T^3$, with the zero set of $\rho_{\phi}$ on $\T^2$ as in Theorem \ref{all pts on component same Lp}. If 
       \[\int_{B_{\varepsilon}(\xi_1,\xi_2)} \sigma(z_1,z_2)^{1-\mathfrak p} \,dm(z_1,z_2) <\infty\]
        for some small $\varepsilon>0$, then 
       \[\int_{B_{\delta}(\xi_1,\xi_2)} \rho_{\phi}(z_1,z_2)^{1-\mathfrak p} \,dm(z_1,z_2) <\infty\]
       for some $0<\delta\leq \varepsilon$.\\
       Specifically, this means that the local $L^{\mathfrak p}_{loc}(\T^3)$-integrability of $\frac{\partial\phi}{\partial z_3}$ at $(\xi_1,\xi_2,\xi_3)$ cannot make the global $L^{\mathfrak p}(\T^3)$-integrability of $\frac{\partial\phi}{\partial z_3}$ worse.
   \end{theorem}
   \begin{proof}
       The function $|\tilde{p}_1|^2$ is continuous with a zero in $(\xi_1,\xi_2)$, so there exists a neighbourhood of $(\xi_1,\xi_2)$ such  that $|\tilde{p}_1|^2$ is bounded. Moreover, the function is bounded on any compact set, so for each $\kappa>0$, there exists an $\varepsilon'>0$ such that $|\tilde{p}_1(z_1,z_2)|^2<\kappa$, for all $(z_1,z_2)\in B_{\varepsilon'}(\xi_1,\xi_2)$. Take $\delta:=\min\{\varepsilon,\varepsilon'\}$. Since $|\sigma|^2\geq 0$, clearly
       \[\int_{B_{\delta}(\xi_1,\xi_2)} \sigma(z_1,z_2)^{1-\mathfrak p} \,dm(z_1,z_2) <\infty\]
       as well. Now we get that
       \begin{align*}
         \int_{B_{\delta}(\xi_1,\xi_2)} \rho_{\phi}(z_1,z_2)^{1-\mathfrak p} \,dm(z_1,z_2) & = \int_{B_{\delta}(\xi_1,\xi_2)} \frac{\sigma(z_1,z_2)^{1-\mathfrak p}}{|\tilde{p}_1(z_1,z_2)|^{2(1-\mathfrak p)}} \,dm \\ 
         & =  \int_{B_{\delta}(\xi_1,\xi_2)} \sigma(z_1,z_2)^{1-\mathfrak p}\, |\tilde{p}_1(z_1,z_2)|^{2(\mathfrak p-1)}  \,dm(z_1,z_2) \\
         & \leq \kappa^{\mathfrak p-1} \int_{B_{\delta}(\xi_1,\xi_2)} \sigma(z_1,z_2)^{1-\mathfrak p} \,dm(z_1,z_2) < \infty
       \end{align*}
       since $\kappa>0$ and $\mathfrak p$ cannot be less than 1. Moreover, since the global integrability is the smallest among all local integrability exponents, the inequality above says that the one of $\frac{\partial\phi}{\partial z_3}$ at $(\xi_1,\xi_2)$ will not be smaller than the global one, and will in turn not affect the global integrability.
   \end{proof}

    The above theorem, together with the analysis done in Example \ref{Example 5.2. global Lp}, yields the critical $L^{\mathfrak p}(\T^3)$ integrability index of $\frac{\partial\phi}{\partial z_3}$ for that example, namely $\mathfrak p^*=\frac{3}{2}$. This addresses Question 2 in \cite[Section 5]{bickel2022singularities}, answering the first part of the question, while completely disproving the second. Not only does the presence of a vertical line singularity fail to worsen the integrability, but locally the integrability is either improved or remains the same, and it has no effect on the global integrability.   
     
   \begin{remark*}
        The claim of the Corollary \ref{corollary 3 equivalences around the vertical line and its component} applies to RIFs with several vertical line singularities as well, since all of them are disjoint, and the argument holds for each of the lines separately. By Theorem \ref{situation better in vertical lines}, this means that none of them have an effect on the global integrability.
    \end{remark*}

    \section{Higher dimensions}{\label{section higher dimensions}}
    This section is dedicated to generalizing results of Section \ref{section on curve singularities} to higher-dimensional irreducible RIFs of polydegree $(m_1,\dots,m_{n-1},1)$ with singularities on $\T^n$. Once again, if an RIF $\phi$ is linear in the last variable, we can write
    \[\phi(z)=\frac{\tilde{p}_1(z_1,\dots,z_{n-1})\cdot z_n+ \tilde{p}_2(z_1,\dots,z_{n-1})}{p_1(z_1,\dots,z_{n-1})+ p_2(z_1,\dots,z_{n-1}) \cdot z_n}.\]
    and therefore define
    \[\rho_{\phi}(z_1,\dots,z_{n-1})=\frac{\sigma(z_1,\dots,z_{n-1})}{|\tilde{p}_1(z_1,\dots,z_{n-1})|^2},\]
    which is analytic at all points of $\T^{n-1}$ except those where $\tilde{p}_1=0$. That again happens at points contained in a vertical line, i.e. we get the immediate result:

    \begin{corollary}{\label{rho_phi analytic n variables}}
        The function $\rho_{\phi}$ is analytic on $\T^n\backslash\bigcup_{\alpha\in A }\{\xi_{\alpha}\}$, where $\xi_{\alpha}\in\T^{n-1}$ are points at which $\phi$ has vertical line singularities $\{\xi_{\alpha}\}\times\T$. Moreover, $\sigma$ is analytic on $\T^{n-1}$.
    \end{corollary}

    In the Corollary above, the set $A$ containing $\alpha\in A$ need not be finite. Unlike the three variable case, general $(m_1,\dots,m_{n-1},1)$-degree RIFs can have infinitely many vertical lines. For example, if the zero set of $p$ on $\T^4$ contained a set of the form $\{(\xi,\xi^2, 1)\}\times\T$ for $\xi\in\T$, this would be called a vertical\textit{ surface} singularity (Figure \ref{vertical surface image}). Therefore, we obtain:

    \begin{corollary}
        For $n\geq 4$, $(m_1,\dots,m_{n-1},1)$-degree RIFs can have infinitely many vertical line singularities. In particular, they are arranged as a finite discrete set and/or as a continuum of vertical lines. 
    \end{corollary}
    \begin{proof}
        Vertical lines appear precisely when $|\tilde{p}_1|^2=0$. Assume $A\subset\mathcal{Z}_{|\tilde{p}_1|^2}\cap\T^{n-1}$ is a closed discrete set. This set cannot have limit points, as they are points whose all neighbourhood intersect some other points of the set. However, $\T^{n-1}$ is compact in the topology inherited from $\mathbb{R}^{n-1}$, so by \cite[Theorem 28.1]{munkres_topology}, every infinite subset of $\T^{n-1}$ has a limit point. Therefore, $A$ must be finite. 
    \end{proof}

    \vspace{-1pc}
    \begin{figure}[h]
		\centering
		\includegraphics[width=4.5cm]{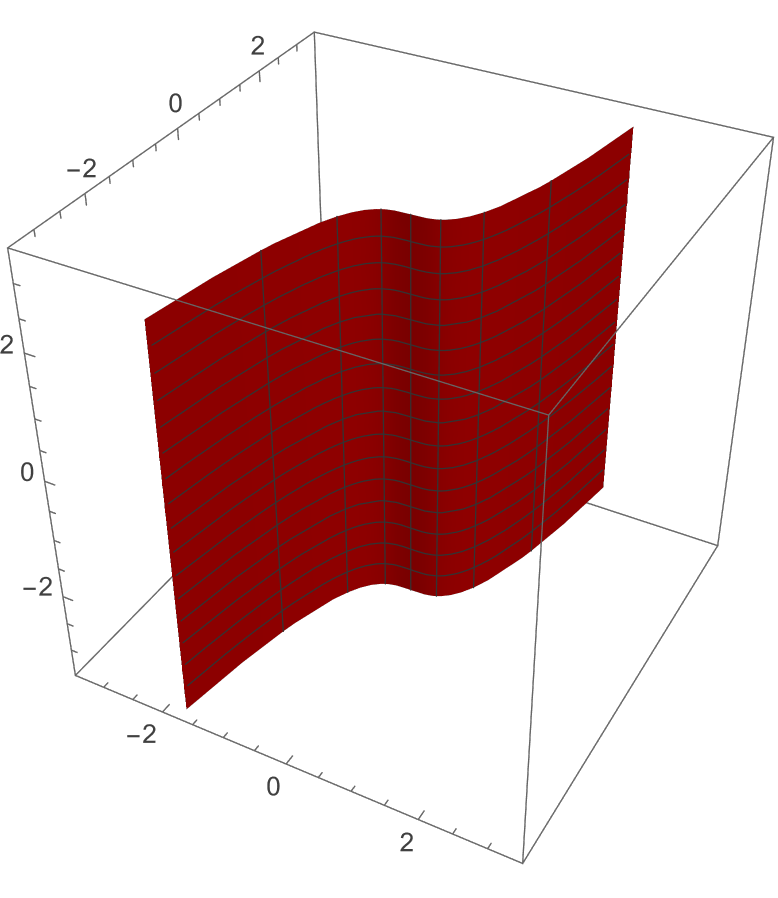}
		\caption{Vertical surface $\{(\xi,\xi^2, 1)\}\times\T$, shown in $\T^3$ by dropping the fixed third coordinate.}
        \label{vertical surface image}
	\end{figure}
    \vspace{-1pc}

    \begin{example}{\label{example 4 variable RIF}}
    We follow the matrix construction as in \cite[Example 5.1.]{bickel2022singularities}, where a three-variable RIF was created. In order to accommodate for an extra variable in a construction of a four-variable function, we pass to a $5\times5$ self-adjoint matrix $A$ and four contractions $Y_1,Y_2,Y_3$ and $Y_4$. Take the matrices as following. 
     
	\begin{align*}
		A &= \begin{pmatrix}
			1 & 0 & 0 & 0 & 0 \\
			0 & 0 & 1 & 0 & 0 \\
            0 & 1 & 0 & 0 & 1 \\
			0 & 0 & 0 & 1 & 0 \\
			0 & 0 & 1 & 0 & 1 
		\end{pmatrix},
		&
		Y_1 &= \begin{pmatrix}
			1 & & & & \\
              & 0 & & & \\
              & & 0 & & \\
              & & & 0 & \\
              & & & & 0
		\end{pmatrix},
		&
		Y_2 &= \begin{pmatrix}
			0 & & & & \\
              & 0 & & & \\
              & & 1 & & \\
              & & & 0 & \\
              & & & & 0
		\end{pmatrix},
	\end{align*}
    \begin{align*}
        Y_3 &= \begin{pmatrix}
			0 & & & & \\
              & 1 & & & \\
              & & 0 & & \\
              & & & 1 & \\
              & & & & 0
		\end{pmatrix}, 
        &
        Y_4 = I-Y_1-Y_2-Y_3= \begin{pmatrix}
			0 & & & & \\
              & 0 & & & \\
              & & 0 & & \\
              & & & 0 & \\
              & & & & 1
		\end{pmatrix}.
    \end{align*}
    Finally, we take $v=(1,1,0,0,1)^T$, after which \cite{agler2016nevanlinna} guarantees that we will obtain a Pick function on the poly-upper half-plane $\Pi^3$:
	\begin{align*}
		f(z) & = \;\langle (A-z_1 Y_1-z_2Y_2-z_3Y_3)^{-1}v,v\rangle, \; \; z\in\Pi^3 \\
		&=\frac{1+z_2-z_1 z_2-z_3-2z_2 z_3+z_1 z_2 z_3-z_4-z_2 z_4+z_1 z_2 z_4 +z_2 z_3 z_4}{(1-z_1)(1-z_3-z_2 z_3-z_4+z_2 z_3 z_4)}.
	\end{align*}
	By taking the composition
	\begin{align*}
		\phi(z_1,z_2,z_3)=(\beta\circ f \circ m) (z_1,z_2,z_3)
	\end{align*}
    for
    \vspace{-0.5pc}
    \[m:\D^3 \to \Pi^3, \quad (z_1,z_2,z_2)\mapsto \bigg(i \,\frac{1-z_1}{1+z_1},i\, \frac{1-z_2}{1+z_2},i \,\frac{1-z_3}{1+z_3}\bigg)\]
    and
    \[\beta:\Pi\to\D, \quad z\mapsto \frac{1+iz}{1-iz},\]
	we obtain a four-variable $(1,1,1,1)$-degree irreducible rational inner function $\phi$ with
	\begin{align*}
	    p_1(z_1,z_2,z_3) &=
         -3i + (2-5i)z_1 -(4-i)z_2 -(2+5i)z_1z_2 +(2+i)z_3 +3iz_1z_3 \\
         &\qquad -(2+3i)z_2z_3 +(4-5i)z_1z_2z_3 \\
        p_2(z_1,z_2,z_3) &= 2-i +(4+i)z_1 -(2+i)z_2 -3iz_1z_2 +3iz_3 -(2-i)z_1z_3\\
        &\qquad+(4-5i)z_2z_3 +(10+5i)z_1z_2z_3) \\
    \tilde{p}_1(z_1,z_2,z_3)&=
    -4-5i +(2-3i)z_1 +3iz_2 -(2-i)z_1z_2 +(2-5i)z_3+(4+i)z_1z_3 \\
    &\qquad-(2+5i)z_2z_3 -3iz_1z_2z_3 \\
    \tilde{p}_2(z_1,z_2,z_3) &=-10+5i -(4+5i)z_1 +(2+i)z_2 +3iz_1z_2 -3iz_3 +(2-i)z_1z_3 \\
    &\qquad -(4-i)z_2z_3-(2+i)z_1z_2z_3.
	\end{align*}

    \noindent
     Direct calculation shows $\mathcal{Z}_{\tilde{p}_1}\cap\T^3=\{(i,-1,1)\}$. Therefore, $\phi$ has only one isolated vertical line singularity on $\T^4$. The complete singular locus is, however, much bigger. In particular, 
    \begin{align*}
        \mathcal{Z}_p\cap\T^4&=\bigg\{\bigg(i,z_2,z_3,\frac{(-5+i)-(1-i)z_2+(1-i)z_3-(3+i)z_2 z_3}{(1+3i)+(1-i)z_2-(1-i)z_3-(1-5i)z_2 z_3}\bigg)\; : \; z_2,z_3\in\T\bigg\}\\
        &\cup \bigg\{\bigg(z_1,1,z_3,\frac{(4+2i)+10iz_1+2iz_3-(4-2i)z_1 z_3}{-2i+(4-2i)z_1+(4-2i)z_3+(8+6i)z_1 z_3}\bigg)\; : \; z_1,z_3\in\T\bigg\} \\
        &\cup \{(z_1,z_2,1,i)\; : \; z_1,z_2\in\T\}
    \end{align*}
    The sets are shown in $\T^3$ in Figure \ref{figure:4 variable zero set}, with the first, second and third variable fixed, respectively. The vertical line $\{(i,-1,1)\}\times\T$ is contained in the first set.

    \begin{figure}[h!]
    \centering
    \begin{subfigure}[t]{0.32\textwidth}
        \centering
        \includegraphics[width=\textwidth]{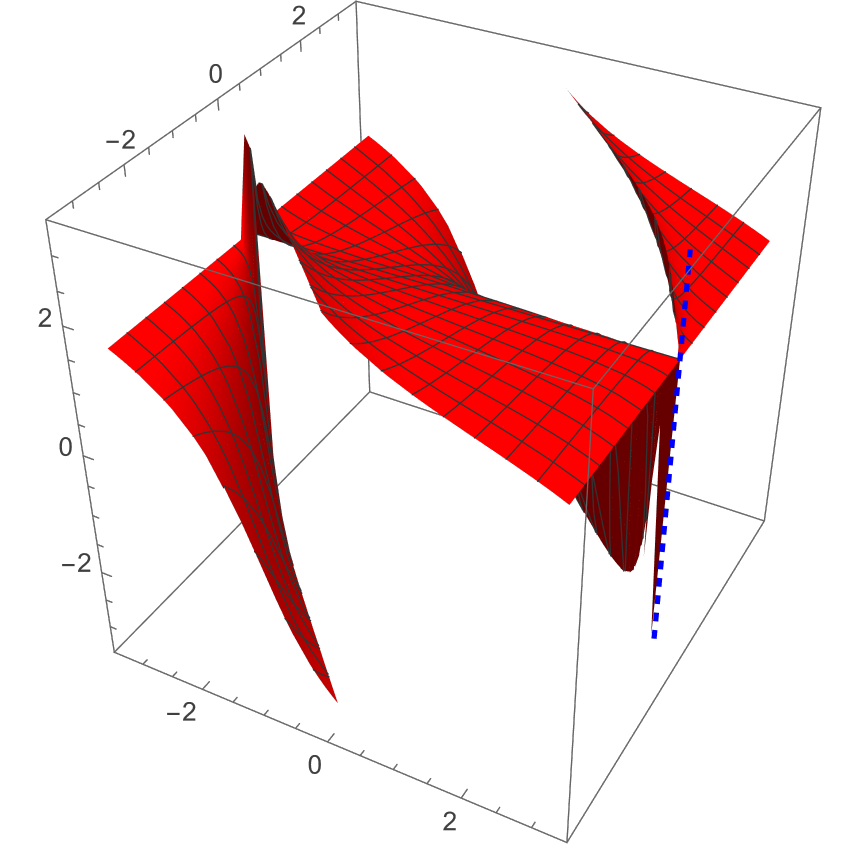}
    \end{subfigure}
    \hfill
    \begin{subfigure}[t]{0.32\textwidth}
        \centering
        \includegraphics[width=\textwidth]{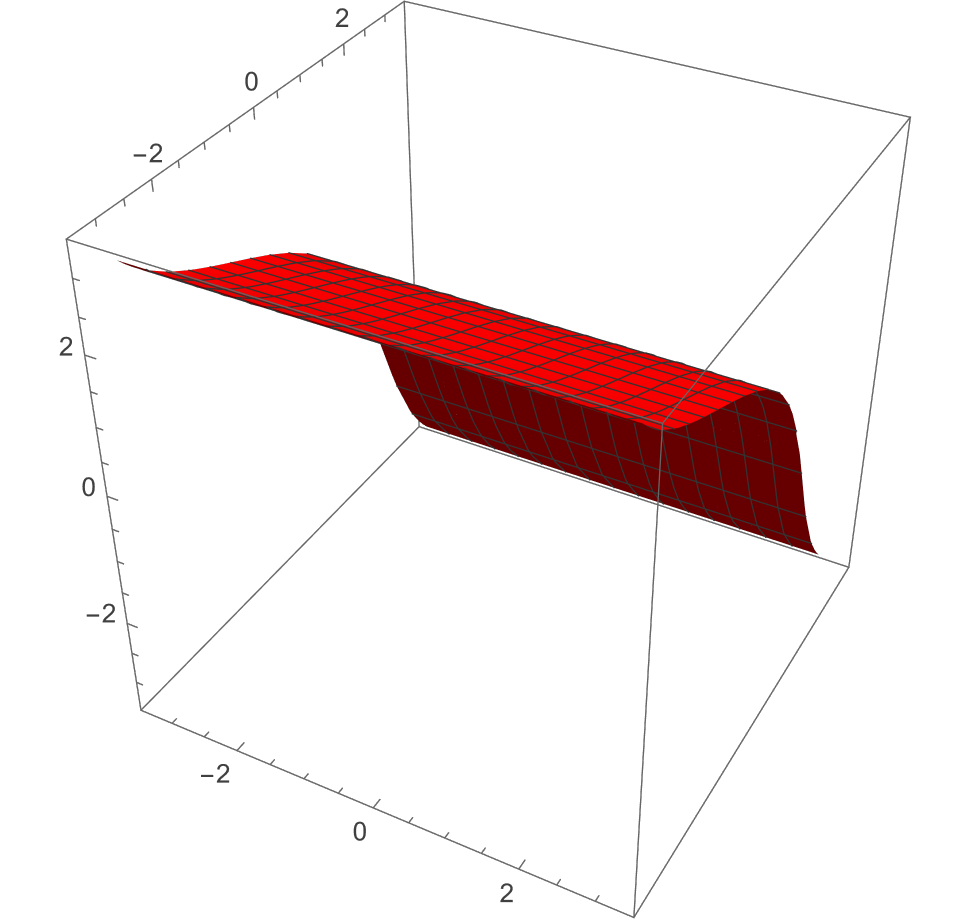}
    \end{subfigure}
    \hfill
    \begin{subfigure}[t]{0.33\textwidth}
        \centering
        \includegraphics[width=\textwidth]{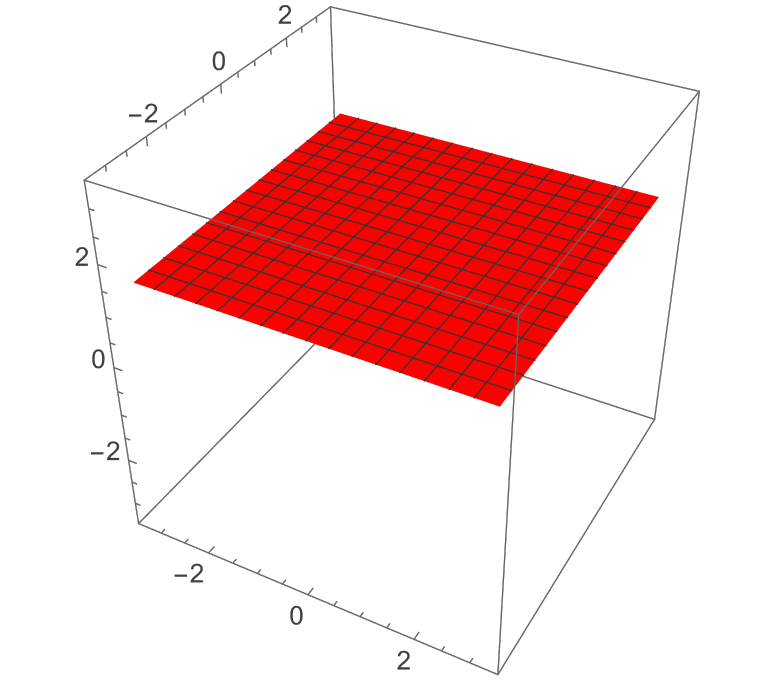}
    \end{subfigure}
    \caption{$\mathcal{Z}_{\phi}\cap\T^4$ from Example \ref{example 4 variable RIF} (pictured after dropping the constant coordinates $z_1,z_2$ and $z_3$, respectively)}
    \label{figure:4 variable zero set}
    \end{figure}
    \end{example}
    \vspace{-1pc}
    
    The complexity of the previous example motivates the extension of previous three-variable results to higher dimensions. Fortunately, this extension is quite straightforward. First, we obtain a generalization of Lemma \ref{Z_p and Z_rho_phi (almost) the same} to $n\geq 3$ variables.
    
    \begin{lemma}{\label{lemma n variables Z phi and Z rho the same}}
        Let $\phi$ be an irreducible $(m_1,\dots,m_{n-1},1)$-degree RIF. Then the following holds. 
        \begin{enumerate}
        \item[(i)] If $\phi$ has no vertical line singularities, $\rho_{\phi}(\zeta_1,\dots,\zeta_{n-1})=0$ if and only if $(\zeta_1,\dots,\zeta_{n-1},\tau)$ is a singularity of $\phi$ for some $\tau\in\T$.
        \item[(ii)] If $\mathcal{Z}_{p}\cap\T^n\supset \bigcup_{\alpha\in A }\{\xi^{\alpha}\}\times\T$, for $\alpha\in A$ from some collection $A$, and $\xi^{\alpha}\in\T^{n-1}$, then $\rho_{\phi}(\zeta_1,\dots,\zeta_{n-1})=0$ if and only if $(\zeta_1,\dots,\zeta_{n-1},\tau)\in\mathcal{Z}_p\cap\T^n\,\backslash\big( \{\xi^{\alpha}\}\times\T\big)$.
        \end{enumerate}  
    \end{lemma}
    \begin{proof}
        The proof follows the same steps as the one of Lemma \ref{Z_p and Z_rho_phi (almost) the same}, since by Corollary \ref{rho_phi analytic n variables}, the function $\sigma$ is analytic at all points except those possibly contained in vertical line singularities of $\phi$. 
    \end{proof}
    
    \begin{example}{\label{example 4 variable rho_phi}}
    Going back to Example \ref{example 4 variable RIF}, we recover the zero set of $\sigma$ (Figure \ref{figure: 4 variable Z_rho_phi}),
    \[\mathcal{Z}_\sigma\cap\T^3=\{(i,z_2,z_3)\,:\, z_2,z_3\in\T\} \cup \{(z_1,1,z_3)\,:\, z_1,z_3\in\T\} \cup \{(z_1,z_2,1)\,:\, z_1,z_2\in\T\},\]
    which is exactly equal to $\mathcal{Z}_{\rho_{\phi}}\cap\T^3\backslash\{(i,-1,1)\}$.

     \begin{figure}[h]
		\centering
		\includegraphics[width=4.5cm]{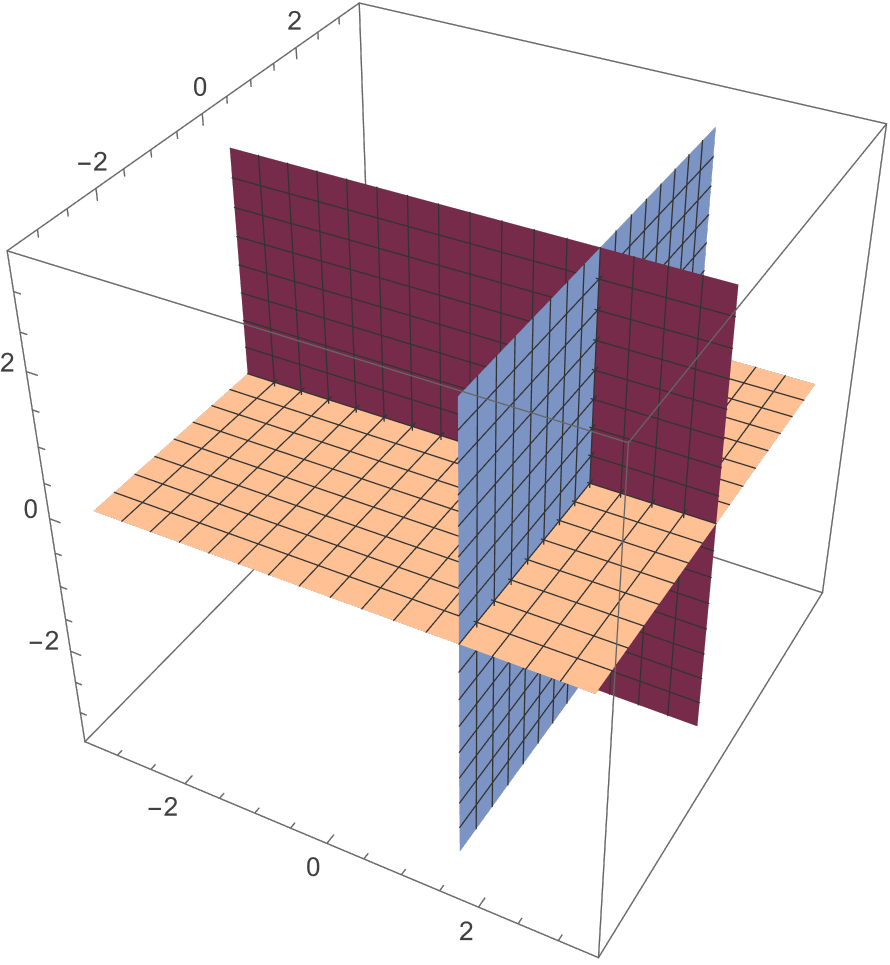}
		\caption{The three sets making up $\mathcal{Z}_\sigma\cap\T^3$}
        \label{figure: 4 variable Z_rho_phi}
	\end{figure}
    \end{example}
    \vspace{-1pc}
    
  In light of this example and previous research, we are naturally led to the question of integrability in higher dimensions. Consequently, we obtain an analogue of Theorem \ref{theorem_rho=p1-p2} and Lemma \ref{lemma_rho}.

    \begin{theorem}{\label{theorem Lp with rho and sigma n variables}}
        Let $\phi$ be a $(m_1,..,m_{n-1},1)$-degree RIF with possible (not necessarily isolated) vertical line singularities $\bigcup_{\alpha\in A }\{\xi^{\alpha}\}\times\T$, for some collection $A\ni \alpha$. Then at all points $\zeta\in\mathcal{Z}_p\cap\T^n$ not contained in $\bigcup_{\alpha\in A }\{\xi^{\alpha}\}\times\T$,
        \begin{align*}
        \frac{\partial\phi}{\partial z_n}\in L^{\mathfrak p}_{loc}(\T^n) &\iff \int_{B_{\varepsilon}(\zeta_1,\dots,\zeta_{n-1})} \rho_{\phi}(z_1,\dots,z_{n-1})^{1-\mathfrak p} \,dm(z_1,\dots,z_{n-1}) <\infty \\
        &\iff \int_{B_{\varepsilon}(\zeta_1,\dots,\zeta_{n-1})} \sigma(z_1,\dots,z_{n-1})^{1-\mathfrak p} \,dm(z_1,\dots,z_{n-1}) <\infty
        \end{align*}
		for sufficiently small $\varepsilon>0$. In particular, if $\phi$ contains no vertical line singularities, the above equivalences hold for all $\zeta\in\mathcal{Z}_p\cap\T^n$.
    \end{theorem}
    \begin{proof}
       The  equivalences follow by using Lemma \ref{lemma n variables Z phi and Z rho the same} and repeating the steps of Theorem \ref{theorem_rho=p1-p2}, when there is no vertical lines, and afterwards, of Lemma \ref{lemma_rho}, when there is.
    \end{proof}

    And accordingly, we once again make use of Theorem \ref{Demailly} to obtain the following result, the generalization of Theorem \ref{all pts on component same Lp}.

    \begin{theorem}{\label{situation better in VL n variable}}
       Let $(\zeta_1,\dots,\zeta_{n-1})\in\T^{n-1}$ be a point in the zero set of $\rho_{\phi}$ around which $\frac{\partial \phi}{\partial z_n}\in L^{\mathfrak p}_{loc}(\T^n)$. Assume $\mathcal{Z}_{\sigma}\cap\T^{n-1}$ is of pure real dimension $n-2$ so that $\mathcal{Z}_{\rho_{\phi}}\cap\T^{n-1}=\bigcup_{i=1} ^N M_i$ for some $N\in\mathbb{N}$, where each $M_i$ is a real manifold of dimension $n-2$, and let $(\zeta_1,\dots,\zeta_{n-1})\in M_j\backslash \bigcup_{k\neq j} M_k$ for some $j\in \{1,\dots,N\}$. Then $\frac{\partial\phi}{\partial z_n}\in L^{\mathfrak p}_{loc}(\T^n)$ at every $(z_1,\dots,z_{n-1})\in M_j\backslash \bigcup_{k\neq j} M_k$.
    \end{theorem}

    It is now clear that the integrability results established in the previous section have all translated nicely to higher dimensions as well. To finalize the analysis, we state the $n$-dimensional extension of Theorem \ref{situation better in vertical lines}.

    \begin{theorem}
        Let $\phi$ be an RIF with a vertical line singularity  $\{\xi=(\xi_1,\dots,\xi_{n-1})\}\times\T\in \mathcal{Z}_p\cap\T^n$ and the zero set of $\sigma$ and $\rho_{\phi}$ as in Theorem \ref{situation better in VL n variable}. If 
        \vspace{-1pc}
       \[\int_{B_{\varepsilon}(\xi)} \sigma(z_1,\dots,z_{n-1})^{1-\mathfrak p} \,dm(z_1,\dots,z_{n-1}) <\infty\]
        for small $\varepsilon>0$, then 
        \vspace{-1pc}
       \[\int_{B_{\delta}(\xi)} \rho_{\phi}(z_1,\dots,z_{n-1})^{1-\mathfrak p} \,dm(z_1,\dots,z_{n-1}) <\infty\]
       for some $0<\delta\leq \varepsilon$.\\
       Specifically, this means that the local $L^{\mathfrak p}_{loc}(\T^n)$-integrability of $\frac{\partial\phi}{\partial z_n}$ at $(\xi,\tau), \tau\in\T$ cannot make the global $L^{\mathfrak p}(\T^n)$-integrability of $\frac{\partial\phi}{\partial z_n}$ worse.
    \end{theorem}
    \begin{proof}
        First, notice that the assumption of a pure dimensional $\mathcal{Z}_{\sigma}\cap\T^{n-1}$ dictates that $\xi$ is contained in a larger component of the zero set of $\sigma$, i.e. that every open neighbourhood around $\xi$ contains infinitely many points of $\mathcal{Z}_{\rho_{\phi}}\cap\T^{n-1}$. Otherwise, $\{\xi\}$ itself would be a component of the zero set of $\sigma$ on $\T^{n-1}$ with dimension $0$. But this is smaller than $n-2$ for any $n\geq 3$ and therefore contradicts the statement of Theorem \ref{situation better in VL n variable}.

        If $\phi$ contains only isolated vertical line singularities (for instance, as in Example \ref{example 4 variable RIF}), we proceed by extending arguments made in Theorem \ref{situation better in vertical lines} to higher dimensions as before. Applying Theorem \ref{situation better in VL n variable}, and redoing analysis developed in Corollary \ref{corollary 3 equivalences around the vertical line and its component} proves the statements. 

        Alternatively, if $\{\xi\}\times\T$ is a part of a larger vertical line set, such as a vertical surface or a vertical threefold, that set also can't be a \q{self-standing} component. In the smallest dimension where such vertical line sets may appear, which is $n=4$, this would mean that we have an isolated curve of points $\{\xi^{\alpha}\}$ of dimension 1 in a $n-2=2$ -dimensional zero set of $\sigma$ on $\T^{n-1}$. Again, contradiction with the statement of Theorem \ref{situation better in VL n variable}. Now, the same argument clearly holds for higher dimensions too. Therefore, reusing the arguments of Corollary \ref{corollary 3 equivalences around the vertical line and its component}, by swapping the curve $\Gamma$ containing a smooth component $\gamma_i$ and the point at which there is a vertical line, with a higher dimensional $\mathcal{M}=M\cup\{\xi^{\alpha}\}$, for $A\subset\mathcal{Z}_{\sigma} \cap\T^{n-1}$, we prove the statement.
    \end{proof}

    \begin{example}
        We apply the previous result to the rational inner function from Example \ref{example 4 variable rho_phi}, whose zero set of $\rho_{\phi}$ on $\T^3$ is composed of three surfaces $A_1=\{(i,z_2,z_3)\,:\, z_2,z_3\in\T\}$, $A_2=\{(z_1,1,z_3)\,:\, z_1,z_3\in\T\}$, $A_3=\{(z_1,z_2,1)\,:\, z_1,z_2\in\T\}$.
        
        Expanding $\sigma$ into a series around points corresponding to all seven possible cases (points lying in a single component, in the intersection of two components, and in the intersection of all three) gave the same result in each case: the critical integrability of $\mathfrak p^*=\frac{3}{2}$. The details are omitted as the calculation is identical to previous examples.

    \end{example}

    \subsection{Other partial derivatives}
    If $\phi$ is an rational inner function that is linear in $j$-th variable $z_j$ instead of (or in addition to) the last variable $z_n$, the theory developed so far for the integrability of $\frac{\partial\phi}{\partial z_n}$  is directly applicable in the analysis of integrability of $\frac{\partial\phi}{\partial z_j}$. 
     Strictly speaking, the main difference will lie in the terminology, since for example, vertical lines would no longer be called \q{vertical}, but instead various axis-parallel lines.

    \section{Compositions}{\label{Compositions}}
    \subsection{The slice matrix $M_{\phi}$}
    In this section, we introduce the notion of a slice matrix and investigate to what extent the properties established in \cite{sola_2023} remain valid in our setting of curve zero sets. 
    
    \vspace{0.5pc}
    \noindent
    First recall that a $(m_1,\dots,m_{n-1},1)$-degree $\phi$ can be written in the form of (\ref{general formula phi}). Now the following matrix representation comes almost naturally.

    \begin{definition}
        The \textit{slice matrix} of a $(m,1)=(m_1,\dots,m_{n-1},1)$-degree RIF is the function $M_{\phi}:\T^{n-1}\to M_{2,2}(\C)$ given by
        \[M_{\phi}(z)= M_{\phi}(z_1,\dots, z_{n-1})= \begin{pmatrix}
            \tilde{p}_1(z) & \tilde{p}_2(z) \\
            p_2(z) & p_1(z)
        \end{pmatrix}.\]
        The \textit{slice determinant} of $\phi$ is the function $P_{\phi}:\T^{n-1}\to \C$ given by
        \[P_{\phi}(z)=\det M_{\phi}(z), \quad z\in\T^{n-1}.\]
    \end{definition}
    
    Given a finite-singularity RIF with polydegree $(m,1)$, it has been shown in \cite[Lemma 2.]{sola_2023} that for $\xi\in\T^{n-1}$, $P_{\phi}(\xi)=0$ if and only if $(\xi,\tau)$ is a singularity of $\phi$ for some value of $\tau\in\T$. Furthermore,
    \vspace{0.2pc}
    \[\frac{\partial\phi}{\partial z_n}\in L^{\mathfrak p}_{loc}(\T^n) \text{ at } (\xi,\tau) \iff \int_{B_{\varepsilon}(\xi)} |P_{\phi}(z)|^{1-\mathfrak p} dm(z) <\infty.\]

    \vspace{0.2pc}
    In the setting of RIFs with higher-dimensional zero sets, the only difference lies in the failure of $P_{\phi}$ to detect vertical line singularities. In particular, we obtain the following statement.

    \begin{proposition}{\label{prop P_phi same as phi}}
        If $\phi$ contains vertical line singularities, then for all $\zeta=(\zeta_1,\dots,\zeta_{n-1})\in\T^{n-1}$, the slice determinant $P_{\phi}(\zeta)=0$ if and only if $(\zeta,\tau)$ is a singularity of $\phi$ not contained in a vertical line singularity, for some $\tau\in\T$. Then around such $\zeta$ 
        \[\frac{\partial\phi}{\partial z_n}\in L^{\mathfrak p}_{loc}(\T^n) \iff  \int_{B_{\varepsilon}(\zeta)} |P_{\phi}(z_1,\dots,z_{n-1})|^{1-\mathfrak p} dm(z_1,\dots,z_{n-1}) <\infty.\]
		for sufficiently small $\varepsilon>0.$
    \end{proposition}
    \begin{proof}
        Since $z_j\bar{z}_j=1$ on $\T$, for all $j=1,2$, we have the equality
        \[p_i(z_1,z_2)=\tilde{\tilde{p}}_i(z_1,z_2)=z_1^{m_1}\cdots z_{n-1}^{m_{n-1}} \,\overline{\tilde{p}_i\Big(\frac{1}{\bar{z}_1},\dots,\frac{1}{\bar z_{n-1}}\Big)}= z_1^{m_1}\cdots z_{n-1}^{m_{n-1}} \,\overline{\tilde{p}_i(z_1,z_2)} \]
        for $i,j=1,2$. Hence, we can rewrite $P_{\phi}$ as
        \[P_{\phi}=\det M_{\phi}=\tilde{p}_1p_1-p_2\tilde{p}_2= z_1^{m_1}\cdots z_{n-1}^{m_{n-1}}(|\tilde{p}_1|^2-|\tilde{p}_2|^2),\]
        and since $z_1^{m_1}\cdots z_{n-1}^{m_{n-1}}$ is non-vanishing on $\T^{n-1}$, we have $\mathcal{Z}_{P_{\phi}}=\mathcal{Z}_{\sigma}$ on $\T^{n-1}$.
        Since\\ 
        $|z_1|^{m_1}\cdots |z_{n-1}|^{m_{n-1}}=1$ is bounded on $\T^{n-1}$, applying Lemma \ref{lemma n variables Z phi and Z rho the same} and Theorem \ref{theorem Lp with rho and sigma n variables} now proves all the claims in their respective order. 
    \end{proof}

    \begin{example}
    Calculation of the slice matrix of the RIF in Example \ref{Example 5.2.},
	\[M_{\phi}(z_1,z_2) = \begin{pmatrix}
		-2-z_1-z_1^2+z_2-z_1z_2+4z_1^2z_2 && 1-z_1-z_1 z_2 +z_1^2z_2 \\
		1-z_1-z_1z_2+z_1^2z_2 && 4-z_1+z_1^2-z_2-z_1z_2-2z_1^2z_2
	\end{pmatrix}\]
	gives the determinant
	\begin{align*}
		\det M_{\phi}(z_1,z_2) & = -9-6z_1^2-z_1^4+6z_2+20z_1^2z_2+6z_1^4z_2-z_2^2-6z_1^2z_2^2-9z_1^4z_2^2\\
		& = -(-3-z_1^2+z_2+3z_1^2z_2)^2.
	\end{align*}
    Since 
    \[\det M_{\phi}(z_1,z_2)=0 \iff z_2(1+3z_1^2) = (3+z_1^2) \iff z_2=\frac{3+z_1^2}{1+3z_1^2},\]  the determinant vanishes precisely at the restriction of zero set of $\mathcal{Z}_p \,\cap \mathbb{T}^3$ to the first two variables, that is, precisely when $\sigma$ vanishes. 
    \end{example}

    This claim is even stronger than it might seem. The matrix $M_{\phi}$ helps determine the local critical integrability around every zero point not contained in a vertical line singularity, but by Theorem \ref{situation better in VL n variable}, that might be all we need. If $\sigma$ has a pure dimensional zero set of dimension $n-2$, the integrability around such points cannot be worse than the integrability of the component of the zero set passing through it, therefore $M_{\phi}$ gives us information about the global integrability index $\mathfrak p^*$ as well.

\subsection{Compositions $\phi^N$ and $M_{\phi}^N$}

   In this section, we study iterations of $\phi$ using the slice matrix $M_{\phi}$. We will observe how the zero sets of $\phi$ on $\T^n$ and $\rho_{\phi}$ on $\T^{n-1}$ evolve and how the integrability behaves under composition. First, let's define what we mean by composition.

   \begin{definition}
       For $(m,1)$-degree RIF $\phi$, we define $\phi^N:\C^n\to\C$ inductively as
       \[\phi^N:=\phi(z_1,\dots,z_{n-1},\phi^{N-1}(z_1,\dots,z_n)),\quad \text{for } N\geq 2.\]
   \end{definition}

   From the discussion in \cite[Section 3]{sola_2023}, we know that $\phi^N$ are rational inner functions as well, and that in the case of finite-singularity RIFs that don't experience a polydegree drop, they preserve the zero set of $\phi$. By \textit{polydegree drop} we mean that  the numerator and denominator of $\phi^N$ share a common factor, which then cancels. For example, see \cite[Example 1]{dynamics_sola_tullydoyle} where $\phi(z_1,z_2)=-(2-z_1z_2-z_1-z_2)/(2-z_1-z_2)$ is shown to have all iterations $\phi^N$ of bidegree $(1,1)$. The following lemma shows that the above property also holds for curve singularities.

   \begin{lemma}{\label{zero set of phi and phi^N the same}}
        Let $\phi$ be a $(m_1,\dots,m_{n-1},1)$-degree irreducible RIF and $\phi^N:=\frac{\tilde{p}^{N}}{p^{N}}$ defined as before, with no polydegree drop. Then $\phi^N$ has the same singular set on $\T^n$ as $\phi$, preserving all point, curve and higher-dimensional singularities, including the vertical lines, if $\mathcal{Z}_p\cap\T^n$ contains any.
   \end{lemma}
   \begin{proof}
       Taking the square of the slice matrix $M_{\phi}$ gives
       \[M^2_{\phi}=\begin{pmatrix}
           \tilde{p}_1 & \tilde{p}_2 \\
           p_2 & p_1
       \end{pmatrix} \begin{pmatrix}
           \tilde{p}_1 & \tilde{p}_2 \\
           p_2 & p_1
       \end{pmatrix} = \begin{pmatrix}
           \tilde{p}_1 \tilde{p}_1+ \tilde{p}_2 p_2 & \tilde{p}_1\tilde{p}_2 + \tilde{p}_2 p_1 \\
           \tilde{p}_1 p_2+ p_1 p_2 & p_2 \tilde{p}_2 +p_1p_1
       \end{pmatrix},\]
       which is precisely the slice matrix of $\phi^2$. This can be checked directly by substituting
       \begin{align*}
            \phi^2(z_1,\dots,z_n):=\phi(z_1,\dots,z_{n-1},\phi(z_1,\dots,z_n))&=\frac{\tilde{p}_1\cdot \frac{z_n\cdot\tilde{p}_1+ \tilde{p}_2}{z_n\cdot p_2 + p_1} + \tilde{p}_2}{p_2\cdot \frac{z_n\cdot\tilde{p}_1+ \tilde{p}_2}{z_n\cdot p_2 + p_1} + p_1}\\
            &= \frac{\tilde{p}_1(z_n\cdot\tilde{p}_1 + \tilde{p}_2)+\tilde{p}_2(z_n\cdot p_2 +p_1)}{p_2(z_n\cdot\tilde{p}_1+\tilde{p}_2)+p_1(z_n\cdot p_2+p_1)} \\
            &=\frac{z_n(\tilde{p}_1\tilde{p}_1+\tilde{p}_2 p_2)+\tilde{p}_1\tilde{p}_2+\tilde{p}_2p_1}{z_n(p_2\tilde{p}_1+p_1p_2)+p_2\tilde{p}_2+p_1p_1}.
       \end{align*}
       By induction, we get that $M_{\phi^N}=M_{\phi}^N$ for all iterations $N\in\mathbb{N}$. Thus,
       \[P_{\phi^N}=P_{\phi}\cdots P_{\phi},\]
       so the zero sets of $P_{\phi^N}$ and $P_{\phi}$ on $\T^{n-1}$ coincide. If $\phi$ is a finite-singularity RIF, the claim follows by \cite{sola_2023}. On the other hand, if $\phi$ has higher-dimensional singularities, Lemma \ref{lemma n variables Z phi and Z rho the same} implies that these are preserved under $N$ iterations as well. This also includes vertical line singularities, which can be seen directly from the form of $M_{\phi^N}$; writing $p^{N}=z_n\cdot p_2^{N}+p_1^{N}$, both $p_1^{N}$ and $p_1^{N}$ are sums and products of $p_1$ and $p_2$, and hence vanish at such points. 
    \end{proof}

    This result makes the analysis of integrability conditions of $\frac{\partial \phi ^N}{ \partial z_n}$ quite simple. We recall from the discussion after Theorem \ref{theorem_rho=p1-p2} that the partial derivative $\frac{\partial \phi}{\partial z_n}\in L^1_{loc}(\T^n)$ for all singularities $\xi\in\T^n$ of $\phi$. Therefore, if $\mathfrak p^*$ is its local $z_n$-derivative integrability index at a singularity $\xi\in\T^n$, we can write 
    \vspace{-0.5pc}
    \[\mathfrak p^*:=1+\mathfrak q^*,\]
    for some $\mathfrak q^*\geq0$. Now we get the following theorem. 
   
    \begin{theorem}{\label{p=1+q/N}}
        Let $\phi$ be a $(m_1,\dots,m_{n-1},1)$-degree irreducible RIF and $\phi^N=\frac{\tilde{p}^{N}}{p^{N}}$ with no polydegree drop. Let $\zeta=(\zeta_1,\dots,\zeta_n)\in\T^n$ be a singularity of $\phi$ not contained in a vertical line singularity. If $\mathfrak p^*=1+\mathfrak q^*$ is the local $z_n$-derivative integrability index at  $\zeta$, then $\phi^N$ has the local $z_n$-derivative integrability index equal to $\mathfrak p^*_{(N)}=1+\frac{\mathfrak q^*}{N}$  at $\zeta$.    
    \end{theorem}
    \begin{proof}
        Proposition \ref{prop P_phi same as phi} gives
        \[\frac{\partial\phi}{\partial z_n}\in L^{\mathfrak p}_{loc}(\T^n) \iff  \int_{B_{\varepsilon}(\zeta_1,\dots,\zeta_{n-1})} |P_{\phi}(z_1,\dots,z_{n-1})|^{1-\mathfrak p} dm(z_1,\dots,z_{n-1}) <\infty\]   
        for $\varepsilon>0$ small enough. By Lemma \ref{zero set of phi and phi^N the same}, the singular sets of $\phi$ and $\phi^N$ coincide on $\T^n$, which together with
        \vspace{-0.5pc}
         \[P_{\phi^N}=P_{\phi}\cdots P_{\phi}\]
         gives
         \[\frac{\partial\phi^N}{\partial z_n}\in L^{\mathfrak p}_{loc}(\T^n) \iff  \int_{B_{\varepsilon}(\zeta_1,\dots,\zeta_{n-1})} |P_{\phi}(z_1,\dots,z_{n-1})|^{N(1-\mathfrak p)} dm(z_1,\dots,z_{n-1}) <\infty\] 
         for small $\varepsilon$. The expression $\mathfrak p^*:=1+\mathfrak q^*$ for $\phi$ comes from $\frac{1-\mathfrak p}{\mathfrak q}\geq -1$, which in turn for $\phi^N$ gives $\frac{N(1-\mathfrak p)}{\mathfrak q}\geq -1$. After simplifying, this means that $\mathfrak p\leq 1+\frac{\mathfrak q}{N}$,
         so $\mathfrak p^*_{(N)}=1+\frac{\mathfrak q^*}{N}$ at $\zeta$.
    \end{proof}

    \begin{remark*}
        Similar to how the local $L^{\mathfrak p}_{loc}(\T^n)$-integrability of $\frac{\partial\phi}{\partial z_n}$ at a point contained in a vertical line doesn't affect the global $L^{\mathfrak p}(\T^n)$-integrability of $\frac{\partial\phi}{\partial z_n}$ (as seen in Theorem \ref{situation better in vertical lines}), this shows that neither does it affect the global integrability of $\frac{\partial\phi^N}{\partial z_n}$.
    \end{remark*}

    \begin{example}
        We will apply the previous claims about compositions of $\phi$ and their integrability on Example \ref{Example 5.2.}. The RIF $\phi^2$ has the numerator
        \begin{align*}
            \tilde{p}^2(z_1,z_2)&=17-10z_1+10z_1^2-2z_1^3+z_1^4-8z_2-8z_1 z_2-12 z_1^2 z_2-4z_1^4 z_2+ \\&+z_2^2+2z_1z_2^2+6z_1^2z_2^2+2z_1^3z_2^2+5z_1^4z_2^2 + \\ 
            &+z_3(2-4z_1+2z_1^2-4z_1z_2-8z_1^2z_2-4z_1^3z_2+2z_1^2z_2^2-4z_1^3z_2^2+2z_1^4z_2^2),
        \end{align*}
        which by further calculation gives
        \[\sigma^2(z_1,z_2)=\frac{(-3-z_1^2+z_2+3z_1^2z_2)^4}{z_1^4z_2^2}\]
        and 
        \vspace{-0.5pc}
        \[P_{\phi^2}(z_1,z_2)=(-3-z_1^2+z_2+3z_1^2z_2)^4.\]
        Thus $\mathcal{Z}_{\sigma^2}\cap\T^2=\mathcal{Z_{\sigma}}\cap\T^2$
        and
        $\mathcal{Z}_{|\tilde{p}_1^{2}|^2}\cap\mathcal{Z}_{\sigma^2}\cap\T^2={\{(1,1)\}}.$
        Either by composing $\phi^2$ further with $\phi$, or by simply multiplying $P_{\phi^2}$ with $P_{\phi}$, we get
        \[\sigma^{N}(z_1,z_2)=\frac{(-3-z_1^2+z_2+3z_1^2z_2)^{2N}}{z_1^{2N}z_2^N}= \frac{P_{\phi^N}(z_1,z_2)}{z_1^{2N}z_2^N}\]
        for all $N\in\mathbb{N}$. Therefore, zero sets of $\phi$ on $\T^3$ and $\rho_{\phi}$, $\sigma$ and $|\tilde{p}_1|^2$ on $\T^2$ are all preserved under iterations, as Lemma \ref{zero set of phi and phi^N the same} claimed.

        \vspace{0.5pc}
         In Example \ref{Example 5.2. global Lp} we have shown that the global critical $z_3$-integrability index of $\phi$ is $\mathfrak p^*=\frac{3}{2}=1+\frac{1}{2}$. By Theorem \ref{p=1+q/N}, it follows that the global $\mathfrak p^*_{(N)}$ of $\phi ^N$ equals to $1+\frac{1}{2N}$, for $N\geq 1$. For example, around $(\theta_1,\theta_2)=(
        \pi,0)$, we have the series expansion (\ref{equation Example 5.2. sigma Lp}), and further,
        \[\sigma^2(\pi-\theta_1,\theta_2)=-256(\theta_1-\theta_2)^4+\mathcal{O}(\|\theta\|^6)\]
        and 
         \[\sigma^3(\pi-\theta_1,\theta_2)=4096(\theta_1-\theta_2)^6+\mathcal{O}(\|\theta\|^8).\]
         Thus $\frac{\partial\phi^2}{\partial z_3},\frac{\partial\phi^3}{\partial z_3}\in L^{\mathfrak p}_{loc}(\T^3)$ around $(\pi,0,\tau)$, $\tau\in\T$ if and only if $\mathfrak p<\frac{5}{4}$ and $\mathfrak p<\frac{7}{6}$, respectively.   
    \end{example}

    \section{A further example}{\label{Examples}}
    
	We construct an example of a type that, to the best of our knowledge, has not appeared previously in literature. In what follows we will see an RIF with two vertical line singularities. Analogous to Example \ref{example 4 variable RIF}, following the matrix construction as in \cite[Example 5.1.]{bickel2022singularities}, we set
	\begin{align*}
		A &= \begin{pmatrix}
			1 & 0 & 1 & 0 \\
			0 & 0 & 1 & 0 \\
			1 & 1 & 0 & 0 \\
			0 & 0 & 0 & 1 
		\end{pmatrix},
		&
		Y_1 &= \begin{pmatrix}
			1 &  &  &  \\
			 & 0 &  &  \\
			 &  & 0 &  \\
			&  &  & 0 
		\end{pmatrix},
		&
		Y_2 &= \begin{pmatrix}
			0 &  &  &  \\
			 & 1 &  &  \\
			 &  & 1 &  \\
			 &  &  & 0 \\
		\end{pmatrix}
        \end{align*}
		and therefore
        \vspace{-1pc}
		\[Y_3 = I - Y_1-Y_2 = \begin{pmatrix}
			0 &  &  &  \\
			 & 0 &  &  \\
			 &  & 0 &  \\
			 &  &  & 1
		\end{pmatrix}.\]
	Here it was necessary for $A$ to be self-adjoint, and $Y_1$ and $Y_2$ to be contractions. By choosing vector $v=(1,0,0,1)^T$, by \cite{agler2016nevanlinna} we obtain the following Pick function on the poly-upper half-plane $\Pi^3$:
	\begin{align*}
		f(w) & = \;\langle (A-w_1 Y_1-w_2Y_2-w_3Y_3)^{-1}v,v\rangle\\
		&=\frac{2-w_1-w_2-2w_2^2+w_1w_2^2-w_3+w_2^2w_3}{(1-w_1-w_2-w_2^2+w_1w_2^2)(1-w_3)},\quad w=(w_1,w_2,w_3)\in\Pi^3.
	\end{align*}
    As before in Example \ref{example 4 variable RIF}, taking the composition
	\begin{align*}
		\phi(z_1,z_2,z_3)=(\beta\circ f \circ m) (z_1,z_2,z_3),
	\end{align*}
	now gives a new RIF $\phi=\frac{\tilde{p}}{p}$, where
	\begin{align*}
	    \tilde{p}(z_1,z_2,z_3)=&-1-4i+(3-4i)z_1+(1-4i)z_2^2+(5-4i)z_1z_2^2\\
        &+z_3\,(3-6i+(7+2i)z_1+(5-2i)z_2^2+(9+6i)z_1z_2^2)\\
        p(z_1,z_2,z_3)=&9-6i+(5+2i)z_1+(7-2i)z_2^2+(3+6i)z_1z_2^2\\
        &+z_3\,(5+4i+(1+4i)z_1+(3+4i)z_2^2-(1-4i)z_1z_2^2).
	\end{align*}
	This rational inner function has two vertical line singularities, namely 
    \[\{(-1,i)\}\times\T\, \cup\, \{(-1,-i)\}\times\T\ \subset \mathcal{Z}_p\,\cap\,\T^3.\]

    \noindent
    The zero set of $\rho_{\phi}$ on $\T^2$ is of the form $(\mathcal{Z}_{\sigma}\cap\T^2) \cup \{(-1,i),(-1,-i)\}\times\T$, where
    \begin{align}{\label{example two vertical lines sigma}}
        \mathcal{Z}_{\sigma}\cap\T^2=\bigg\{\bigg(\frac{-2+3i+(i-2)z_2^2}{2+i+(2+3i)z_2^2},z_2\bigg):z_2\in\T\bigg\}.
    \end{align}
    The entirety of $\mathcal{Z}_{\rho_{\phi}}\cap\T^2$ is composed of two smooth curves; see Figure \ref{fig:2 vertical lines}. However, the zero set of $\sigma$ on $\T^3$ is a smooth curve, therefore it is enough to look at the local derivative integrability around any point of the zero set (\ref{example two vertical lines sigma}). Calculation shows that $\frac{\partial \phi}{\partial z_3}\in L^{\mathfrak{p}}(\T^3)$ for $\mathfrak{p}< \frac{3}{2}$.

    \begin{figure}[h]
	   \centering
		\includegraphics[width=4.5cm]{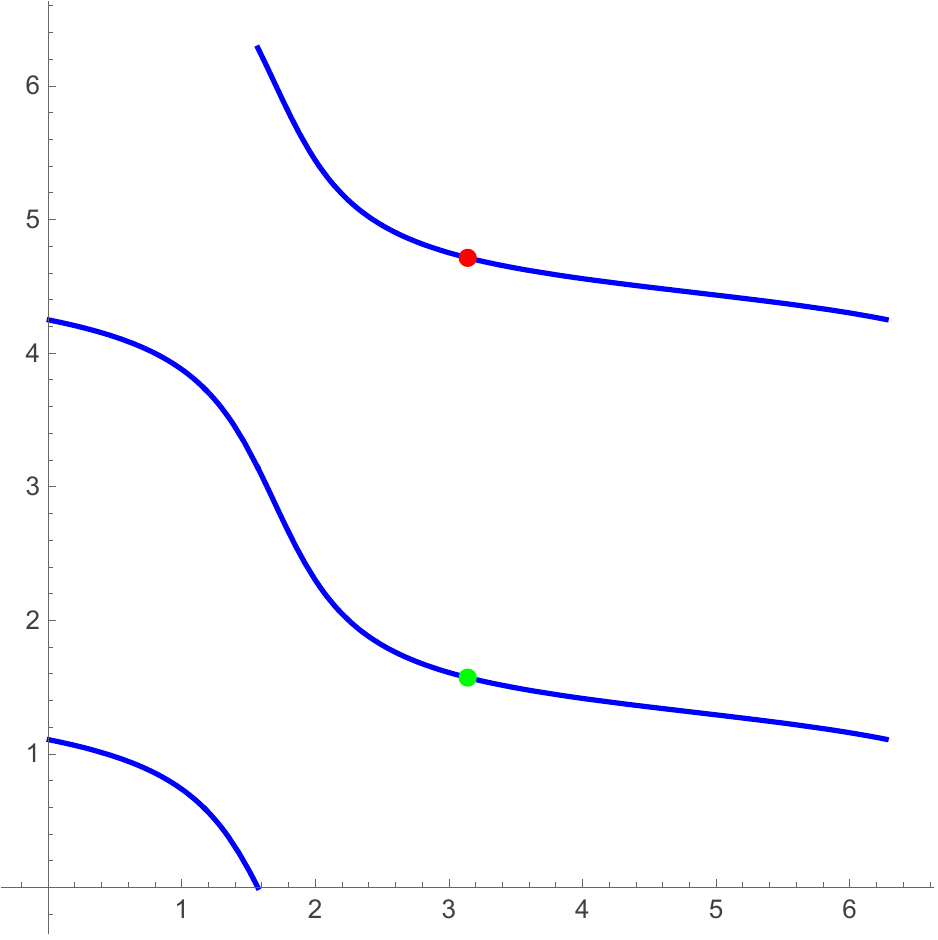}
		\caption{Zero set of $\rho_{\phi}$ shifted for visibility to $(0,2\pi]^2$, with singular points $(\pi,\frac{\pi}{2})$ and $(\pi,\frac{3\pi}{2})$.}
		\label{fig:2 vertical lines}
	\end{figure}
    \noindent
    In order to study the iterations $\phi^N$, we raise $P_{\phi}$ to the $N$-th power. Since
    \[P_{\phi}(z_1,z_2)=\frac{(2-3i+(2+i)z_1+(2-i)z_2^2+(2+3i)z_1z_2^2)^2}{z_1 z_2^2},\]
    we obtain
    \[P_{\phi^N}(z_1,z_2)=\frac{(2-3i+(2+i)z_1+(2-i)z_2^2+(2+3i)z_1z_2^2)^{2N}}{(z_1 z_2^2)^N}.\]
    The zero set of $\sigma$, that is, of $P_{\phi}$, stays preserved under the composition of the function with itself by Lemma \ref{zero set of phi and phi^N the same}. Using the notation of Theorem \ref{p=1+q/N}, we have $\mathfrak{q}^*=\frac{1}{2}$. Then, that same theorem will give us the global critical integrability index of $\phi^N$, notably $\mathfrak{p}^*_N=1+\frac{1}{2N}$.

    \section*{Acknowledgements}
    \noindent
    This material has been adapted from the author's Master's thesis at Stockholm University. 
    The author extends her deepest gratitude to Alan Sola, for his guidance, invaluable advice and unwavering support.

    \addcontentsline{toc}{section}{References}
	\bibliographystyle{plain}
	\bibliography{refs}
\end{document}